\documentclass{article}
\usepackage{amsmath}
\usepackage{amsfonts}
\usepackage{amsthm}
\usepackage{enumerate}

\newtheorem{thm}{Theorem}
\newtheorem{prop}[thm]{Proposition}
\newtheorem{lem}[thm]{Lemma}
\newtheorem{cor}[thm]{Corollary}
\newtheorem{conj}[thm]{Conjecture}
\newtheorem*{thmnn}{Theorem}

\theoremstyle{definition}

\newtheorem{rmk}[thm]{Remark}

\begin{document}

\title{On Differentiating Symmetric Functions}

\author{Shaul Zemel}

\maketitle

\section*{Introduction}

Consider a function $\phi$ of $N$ variables $x_{i}$, $1 \leq i \leq N$, say from a field $\mathbb{F}$, and assume that $\phi$ is symmetric. It is well-known that the elementary symmetric functions $e_{r}:=\sum_{|I|=r}\prod_{i \in I}x_{i}$ (where the sum is over the subsets $I$ of size $r$ of the indices between 1 and $N$) generate the ring of symmetric polynomials in the $x_{i}$'s, and therefore $\phi$ can be presented as a function $\psi$ of the expressions $e_{r}$, $1 \leq r \leq N$, viewed as coordinates on the $N$th symmetric power $\operatorname{Sym}^{N}\mathbb{F}$, which is the quotient of $\mathbb{F}^{N}$ under the action of the symmetric group $S_{N}$ by interchanging the coordinates.

Assume now that $\mathbb{F}$ is the real field $\mathbb{R}$ and $\phi$ is continuously differentiable, or that $\mathbb{F}$ is $\mathbb{C}$ and $\phi$ is holomorphic. Then so will the function $\psi$ of the symmetric expressions $e_{r}$ be, and one may ask how to write the derivatives of the latter function with respect to these coordinates, in terms of the derivatives of $\phi$. Of course, the derivatives of $\phi$ are obtained from those of $\psi$ by the chain rule, but we are interested in inverting this relation, which will, in particular, give a symmetric function of the $x_{i}$'s.

Such questions do show up in some applications, for example in the theory of Riemann surfaces, Abel--Jacobi maps, and Thomae formulae. In proving Thomae's derivative formula (see, e.g., \cite{[EKZ]} for a rather recent case), one has to differentiate the Abel--Jacobi map as a map on positive divisors of fixed degree, which is the same as a symmetric power of the Riemann surface (hence locally a symmetric power of $\mathbb{C}$). The Jacobian of this map involves, at the double points from \cite{[EKZ]}, derivatives of second order (this is already hinted in the proof of the Riemann--Roch theorem in \cite{[FK]}, but without any details). Now, the Abel--Jacobi map is the sum of the values of a single-variable function on the coordinates (such functions are known as \emph{trace functions} in \cite{[B]} and others), which makes a few statements simpler. We will, however, analyze the actions on general functions in this paper.

Since there will be many derivatives in this paper, we shall denote, when $f$ is a function of many variables including $v$, the derivative $\frac{\partial f}{\partial v}$ by simply $f_{v}$. When the variables of $f$ are numbered, say $v_{l}$, $1 \leq l \leq M$, we shall shorthand $f_{v_{l}}$ to simply $f_{l}$. Similarly, we shall write $\partial_{v}$ for the operator $\frac{\partial}{\partial v}$, and with numbered variables we shall use the shortened notation $\partial_{l}$ for $\partial_{v_{l}}$ in case no confusion can arise. We emphasize that in every such derivative the other variables remain constant when we differentiate, and when we express the same function in terms of different coordinates, we will change the name of the function, like the distinction we already made between $\phi$ and $\psi$ above already exemplifies.

\smallskip

Now, as when $N=1$ there is nothing to consider, we examine the case of $N=2$, with the variables $x=x_{1}$ and $y=x_{2}$ and the symmetric functions $s=e_{1}=x+y$ and $p=e_{2}=xy$. Writing the symmetric function $\phi=\phi(x,y)$ as $\psi=\psi(s,p)$, we obtain \[\phi_{x}=\psi_{s}+y\psi_{p}\qquad\mathrm{and}\qquad\phi_{y}=\psi_{s}+x\psi_{p},\] from which we get, when $y \neq x$, the expressions \[\psi_{s}=\frac{y\phi_{y}-x\phi_{x}}{y-x}\qquad\mathrm{and}\qquad\psi_{p}=\frac{\phi_{x}-\phi_{y}}{y-x}.\] Our first observation is that the differential operators that we have to apply on $\psi$ for obtaining these derivatives are, in general, \emph{not} in the Weyl algebra $\mathbb{C}[x,y]\langle\partial_{x},\partial_{y}\rangle^{S_{2}}$ from \cite{[B]} and others, as it involves division by $y-x$. The reason for this is that while for general functions $f$ there is no relation between the derivatives $f_{x}$ and $f_{y}$, for a symmetric function $f$ the difference $f_{x}-f_{y}$ is anti-symmetric, thus vanishes when $x=y$, which means that if $f$ is a polynomial then this difference is divisible by $y-x$. This means that the algebra acting on symmetric functions should be larger than the $S_{N}$-invariants of the ordinary Weyl algebra, and contain also elements like $\frac{\partial_{x}-\partial_{y}}{y-x}$ (which also equals $\frac{\partial_{y}-\partial_{x}}{x-y}$ and is thus symmetric) in two variables. The full Weyl algebra, for any (finite) number of variables, is determined in Corollary \ref{Weyl} below.

Another observation is exemplified by considering the points with $y=x$, for which we take the limit $y \to x$ of our expressions for $\psi_{p}$ and $\psi_{s}$. Since $\phi$ is symmetric, we know that at diagonal points we have $\phi_{x}(x,x)=\phi_{y}(x,x)$, and we find that \[\lim_{y \to x}\psi_{p}=\lim_{y \to x}\frac{\phi_{x}(x,y)-\phi_{x}(x,x)-\phi_{y}(x,y)+\phi_{y}(x,x)}{y-x}=\phi_{xy}-\phi_{yy}\] (the latter summand also equals $\phi_{xx}$ by symmetry, so that $\psi_{p}$ at a diagonal point can be described as the difference between the mixed second derivative of $\phi$ and the pure one, with \emph{the} pure one being independent of the choice of variable). Since $\psi_{s}$ equals $\phi_{y}-x\psi_{p}$ and $\phi_{x}-y\psi_{p}$ (or the more symmetric expression $\frac{\phi_{x}+\phi_{y}}{2}-\frac{x+y}{2}\psi_{p}$), its value at the diagonal point $(x,x)$, where $\phi_{x}=\phi_{y}$, is $\phi_{y}-x(\phi_{xy}-\phi_{yy})=\phi_{x}-x(\phi_{xy}-\phi_{xx})$. Note that these expressions contain the value of the variable itself ($x$ or $y$), and mix derivatives of $\phi$ of different orders.

\smallskip

One might then pose the question, whether we can replace the $e_{r}$'s by a set of coordinates in which the derivatives of $\psi$ will be more natural expressions in the derivatives of $\phi$, in particular involving derivatives of $\phi$ in as uniform a degree as possible when the values of the variables coincide, and without having to multiply by these values themselves. We would like the $r$th coordinates to be homogenous of degree $r$ in the $x_{i}$'s (so that in particular the first one will be $e_{1}$, up to constant multiples), thus expressible in terms of partitions of $r$.

The paper \cite{[B]} considers some differential operators and coordinates, and the ones chosen there are the Newton polynomials of the elementary symmetric functions, with which $e_{r}$ is, up to a multiplicative constant perhaps, the $r$th power sum function $p_{r}:=\sum_{i=1}^{N}x_{i}^{r}$. However, consider the case $N=2$ again, choose again $s$ as the first coordinate but the power sum $q=x^{2}+y^{2}$ as the second one, and write $\phi(x,y)$ as $\eta(s,q)$. Then differentiation yields \[\phi_{x}=\eta_{s}+2x\eta_{q}\qquad\mathrm{and}\qquad\phi_{y}=\eta_{s}+2y\eta_{q},\] from which for $x \neq y$ one obtains \[\eta_{s}=\frac{y\phi_{x}-x\phi_{y}}{y-x}\qquad\mathrm{and}\qquad\eta_{q}=\frac{\phi_{y}-\phi_{x}}{2(y-x)},\] the limiting process at the diagonal is established by noticing that $\eta_{q}$ is $-\frac{1}{2}\psi_{p}$ from before, and $\eta_{s}$ is $\frac{\phi_{x}+\phi_{y}}{2}+(x+y)\eta_{q}$, with again a limit involving derivatives of both orders 1 and 2 at diagonal points.

The right coordinates for $N=2$ are given in Lemma 2.3 of \cite{[EKZ]}. With $s=x+y$ as before but $u=\frac{(y-x)^{2}}{2}$, and with $\phi(x,y)=\varphi(s,u)$, we get \[\phi_{x}=\varphi_{s}-(y-x)\varphi_{u}\qquad\mathrm{and}\qquad\phi_{y}=\varphi_{s}+(y-x)\varphi_{u},\] relations whose inversion yields, when $x \neq y$, the expressions  \[\varphi_{s}=\frac{\phi_{x}+\phi_{y}}{2}\qquad\mathrm{and}\qquad\varphi_{u}=\frac{\phi_{y}-\phi_{x}}{2(y-x)}.\] With these coordinates we now see that $\varphi_{s}$ no longer involves the denominator $y-x$, so that it keeps its form also when $y=x$ (and then it equals $\phi_{x}$ and $\phi_{y}$ as well), and involves only first derivatives of $\phi$ also at such points (the form of the other derivative, $\varphi_{u}$, as a difference of the pure and mixed second derivatives of $\phi$ at the diagonal is preserved, up to a constant multiplier). These are the type of coordinates that we look for in general.

\smallskip

With the coordinates $s$ and $u$ from above, the first derivative that we obtained was a multiple of $\partial_{x}+\partial_{y}$, so that for general $N$ we would like our first derivative to be $\sum_{i=1}^{N}\partial_{i}$. The other derivative (in all three systems of coordinates) was a multiple of $\frac{\partial_{x}-\partial_{y}}{y-x}$, which is symmetric with respect to interchanging $x$ and $y$, and for obtaining a symmetric operator in $N$ variables we take the sum $\sum_{i<j}\frac{\partial_{i}-\partial_{j}}{x_{j}-x_{i}}$ (when some of the values of the $x_{i}$'s coincide, the corresponding expressions can be replaced by the associated combination $\partial_{i}\partial_{j}-\partial_{i}^{2}=\partial_{i}\partial_{j}-\partial_{j}^{2}$ of second derivatives). In three variables $x$, $y$ and $z$, the operator $(y-z)\partial_{x}+(z-x)\partial_{y}+(x-y)\partial_{z}$ is anti-symmetric with respect to interchanging any pair from the these variables. Therefore a symmetric operator is obtained by dividing it by the product $(x-y)(x-z)(y-z)$, which yields $\frac{\partial_{x}}{(y-x)(z-x)}+\frac{\partial_{y}}{(x-y)(z-y)}+\frac{\partial_{z}}{(x-z)(y-z)}$. Summing over all triples with $i<j<k$ of that operator obtained from $x_{i}$, $x_{j}$, and $x_{k}$ then yields a symmetric operator in $N$ variables, which only involves first partial derivatives and division by two differences of variables. When $x_{i}=x_{j}=x_{k}$ the corresponding operator produces a combination of third derivatives, which can be written as $\partial_{i}\partial_{j}\partial_{k}-\frac{3}{2}\partial_{i}^{2}\partial_{j}+\frac{1}{2}\partial_{i}^{3}$ (or any expression obtained by the action of the symmetric group $S_{3}$ on some of the summands).

We can now state our result (Theorem \ref{coorddef} below).
\begin{thmnn}
For every $1 \leq r \leq N$ there is a symmetric polynomial $u_{r}$ in the variables $x_{i}$, $1 \leq i \leq N$, which is homogenous degree $r$, such that $\{u_{r}\}_{r=1}^{N}$ is the unique system of coordinates with the following property: If we consider a symmetric function $\phi$ of the variables $x_{i}$, $1 \leq i \leq N$, then the derivative $\varphi_{r}=\varphi_{u_{r}}$ equals $\sum_{\{i_{h}\}_{h=1}^{r}\in\mathbb{N}_{N}^{r}\mathrm{\ distinct}}\sum_{h=1}^{r}\phi_{i_{h}}\big/\prod_{g \neq h}(x_{i_{g}}-x_{i_{h}})$.
\end{thmnn}
Note that for $N=2$ the variables from our theorem are $u_{1}=\frac{s}{2}$ and $u_{2}=-\frac{u}{4}$, for $\varphi_{u_{1}}$ to be $\phi_{x}+\phi_{y}$ (without the denominator 2) and $\varphi_{u_{2}}$ to be $\frac{\phi_{x}-\phi_{y}}{y-x}+\frac{\phi_{y}-\phi_{x}}{x-y}$, namely $2\frac{\phi_{x}-\phi_{y}}{y-x}$, with the correct sign and the extra multiple (coming from the indices in the operator being distinct, and not necessarily in increasing order). This will be the normalization with which our argument will work in the optimal way (see Remark \ref{norm} below). Note that our operators do not involve multiplications by non-constant functions the values of the variables, as desired. We shall also deduce the formulae for the derivatives at points where the values of some (or all) variables coincide, involving higher order derivatives of $\phi$---see Theorems \ref{expordp}, Proposition \ref{Ddallpt}, and Corollary \ref{totdiag} below.

The coordinates that our main theorem produces have many interesting properties: They can be expressed in terms of the ordinary Bell polynomials (Remark \ref{Bell}---see also Remark \ref{Bellexp} for a similar statement, with the exponential Bell polynomials, for the forms of our derivatives along the total diagonal); For $r\geq2$ they vanish to a high order along the total diagonal (Theorem \ref{annbyders}); They are characterized up to scalar multiples by not having non-trivial numerators in the derivatives with respect to them (Theorem \ref{uniqders}); And as functions of the number $N$ of variables, they are characterized by all of their mixed derivatives decaying faster than usual (Theorem \ref{ordwithN}). In fact, we conjecture that the more mixed the derivative is, the faster it decays with $N$---see Conjecture \ref{decay} below for the precise statement. Note that the Newton polynomials and power sum functions do show up as the limits of these symmetric functions as $N\to\infty$ (see Proposition \ref{Ntoinfty}), but not for the finite values of $N$.

In fact, for evaluating the derivatives at specific points in the symmetric space, we merge our coordinates with the simpler approach of differentiating variables according to their values, and this produces the simplest form of the derivatives around a given point---see Theorem \ref{coorgenpt} below.

\smallskip

The paper is divided into three sections. Section \ref{NSFDer} defines the differential operators and symmetric functions that we need at points where all the variables take distinct values, and Section \ref{PropCoord} investigates some of their properties. Then Section \ref{Diag} considers the limits of the resulting derivatives at points where variables coincide (for obtaining higher order derivatives of $\phi$), and presents the form of our differential operators at any point.

\smallskip

I am grateful to D. Barlet and Y. Kopeliovich for stimulating discussions on this topic, as well as to A. Zemel for suggesting the use of L'H\^{o}pital's Rule instead of the Taylor expansion for examining cases in an earlier draft, which helped obtain the expressions with higher derivatives in general.

\section{Normalized Symmetric Functions and Derivatives \label{NSFDer}}

Let $\mathbb{N}_{N}$ denote the set of integers $i$ with $1 \leq i \leq N$. We recall again that for $1 \leq h \leq N$, the expression $e_{h}:=\sum_{J\subseteq\mathbb{N}_{N},\ |J|=h}\prod_{i \in J}x_{i}$ is the $h$th elementary symmetric function in the variables $x_{i}$, $i\in\mathbb{N}_{N}$ (with $e_{0}$ being the constant function 1), which is homogenous of degree $h$. Thus, if $\lambda$ is a partition of an integer $r$ as the sum of the positive integers $h_{q}$, $1 \leq q \leq l$, with $h_{q} \leq N$ for every $q$, then $e_{\lambda}:=\prod_{q=1}^{l}e_{h_{q}}=\prod_{h=1}^{N}e_{h}^{m_{h}}$ (where $m_{h}$ is the multiplicity of $h$ in $\lambda$) is a symmetric polynomial in the $x_{i}$'s which is homogenous of degree $r$. The number $l$ of (non-zero) integers $h_{q}$ that participate in $\lambda$, which also equals $\sum_{h}m_{h}$, is the \emph{length} $l(\lambda)$ of the partition, and if we denote the \emph{conjugate}, or \emph{transpose}, partition of $\lambda$, obtained by reflecting its Ferrers diagram along the diagonal, by $\lambda^{t}$, then our assumption that the numbers $h_{q}$ participating in $\lambda$ do not exceed $N$ can be written as the inequality $l(\lambda^{t}) \leq N$. The fact that $\lambda$ is a partition of $r$ will henceforth be written as $\lambda \vdash r$. We then recall the following result from the theory of symmetric functions.
\begin{thm}
The products $\{e_{\lambda}\}_{\lambda \vdash r,\ l(\lambda^{t}) \leq N}$ form a basis, over $\mathbb{Z}$, for the symmetric homogenous polynomials of degree $r$ in the $x_{i}$'s having integral coefficients. The ring $\mathbb{Z}[x_{1},\ldots,x_{N}]^{S_{N}}$ of symmetric polynomials in these variables is generated, as a $\mathbb{Z}$-algebra, by $\{e_{h}\}_{h=1}^{N}$. \label{egen}
\end{thm}
Of course, the more natural version of Theorem \ref{egen} is for symmetric functions of infinitely many variables, where the parameter $N$ is omitted, but for our purposes we shall need the version in finitely many variables. Our version is, of course, a consequence of the version with infinitely many variables, when one substitutes all the variables $x_{i}$ with $i>N$ to be 0. Then $e_{h}$ with $h \leq N$ gives the finite versions, while for $h>N$ we get $e_{h}=0$ (since there is no subset of size $h$ in a set of size $N$), and with it $e_{\lambda}=0$ when $l(\lambda^{t})>N$ because of the existence of a vanishing multiplier. We shall thus henceforth adopt the convention for the elementary symmetric function $e_{l}$ in any number $d$ of variables to give the expected expression when $l \leq d$, and to vanish in case $l>d$. For more on partitions and symmetric functions, including a proof of Theorem \ref{egen}, see, e.g., \cite{[M]}.

\smallskip

We shall soon differentiate the elementary symmetric functions, as well as their products from Theorem \ref{egen}, in various ways. For doing so we shall need the following presentation of these functions.
\begin{lem}
Let $I\subseteq\mathbb{N}_{N}$ be a subset of size $d$, and take some $1 \leq h \leq N$. Then we have the expansion \[e_{h}(x_{1},\ldots,x_{N})=\textstyle{\sum_{l=0}^{h}e_{l}\big(\{x_{i}\}_{i \in I}\big) \cdot e_{h-l}\big(\{x_{i}\}_{i \not\in I}\big)}.\] \label{expeh}
\end{lem}
For simplifying the notation, given a set of indices $J$ and an integer $r$, we can write $e_{r}(x_{J})$ for $e_{r}\big(\{x_{i}\}_{i \in J}\big)$. Then Lemma \ref{expeh} can be written as the equality $e_{h}(x_{\mathbb{N}_{N}})=\sum_{l=0}^{h}e_{l}(x_{I})e_{h-l}(x_{\mathbb{N}_{N} \setminus I})=\sum_{l=0}^{h}e_{l}(x_{I})e_{h-l}(x_{I^{c}})$ (where $I^{c}$ is the complement of $I$ in the natural ambient set $\mathbb{N}_{N}$). Moreover, if $h>d$ then the summand $e_{l}(x_{I})$ vanishes for every $d<l \leq h$, so that the sum is effectively taken over $0 \leq l\leq\min\{d,h\}$. However, we shall not be needing the latter fact, and allow for vanishing terms to show up in our expressions.

\begin{proof}
Every subset $J\subseteq\mathbb{N}_{N}$ of size $h$ is the disjoint union of $J \cap I$, of some size $0 \leq l \leq h$, and $J \cap I^{c}$, of size $h-l$. Conversely, fix $0 \leq l \leq h$, and then from every subset $K \subseteq I$ of size $l$, and every subset $H \leq I^{c}$ of size $h-l$, their union $H \cup K$ is a subset of size $h$ of $\mathbb{N}_{N}$. Moreover, the product $\prod_{i \in K}x_{i}$ shows up in $e_{l}(x_{I})$, the product $\prod_{i \in H}x_{i}$ appears in $e_{h-l}(x_{I^{c}})$, and multiplying them gives the summand $\prod_{i \in H \cup K}x_{i}$ from $e_{h}(x_{\mathbb{N}_{N}})$. Now, summing over $H$ and $K$ independently produces all the summands in $e_{h}(x_{\mathbb{N}_{N}})$ that are associated with subsets $J\subseteq\mathbb{N}_{N}$ of size $h$ such that $J \cap I$ has size $l$, and by summing over $0 \leq l \leq h$ we obtain all the sets $J$ of size $h$, and therefore the desired expression $e_{h}(x_{\mathbb{N}_{N}})$. This proves the lemma.
\end{proof}
Note that the summand with $l=0$ in Lemma \ref{expeh} is $e_{h}(x_{I^{c}})$ since $e_{0}(x_{I})=1$ (this matches the summands with empty $H$ in the proof).

\smallskip

When we let a differential operator $D$ act on $e_{h}=e_{h}(x_{\mathbb{N}_{N}})$, if $I$ is the set of variables whose respective derivatives appear in $D$, we can use the decomposition from Lemma \ref{expeh} for evaluating $De_{h}$, by letting $D$ only act on the summands $e_{l}(x_{I})$. Thus, for example, if $D=\partial_{i}$ then we can take $I=\{i\}$, write $e_{h}$ with $h\geq1$ as $x_{i}e_{h-1}(x_{\{i\}^{c}})+e_{h}(x_{\{i\}^{c}})$, and easily obtain that $\partial_{i}e_{h}=e_{h-1}(x_{\{i\}^{c}})$. As another example, let now $D$ be the operator $\frac{\partial_{i}-\partial_{j}}{x_{j}-x_{i}}$, like the one appearing in the Introduction. We then take $I$ to be $\{i,j\}$, and Lemma \ref{expeh} expresses $e_{h}$ for $h\geq2$ as $x_{i}x_{j}e_{h-2}(x_{I^{c}})+(x_{i}+x_{j})e_{h-1}(x_{I^{c}})+e_{h}(x_{I^{c}})$, since for our $I$ we have $e_{1}(x_{I})=x_{i}+x_{j}$ and $e_{2}(x_{I})=x_{i}x_{j}$. Now, our operator $D$ annihilates the last summand, and as its numerator $\partial_{i}-\partial_{j}$ takes $x_{i}+x_{j}$ to 0 and $x_{i}x_{j}$ to $x_{j}-x_{i}$, we deduce that $De_{h}=e_{h-2}(x_{I^{c}})$. While the algebra involved is more tedious, one can verify that in the three variables $x$, $y$, and $z$, the operator $\frac{\partial_{x}}{(y-x)(z-x)}+\frac{\partial_{y}}{(x-y)(z-y)}+\frac{\partial_{z}}{(x-z)(y-z)}$ takes the elementary symmetric functions $e_{0}(x,y,z)=1$, $e_{1}(x,y,z)=x+y+z$, and $e_{2}(x,y,z)=xy+xz+yz$ to 0 and sends the product $e_{3}(x,y,z)=xyz$ to 1, which means that if $I\subseteq\mathbb{N}_{N}$ is a set of size 3 and $D$ is the operator constructed from $I$ in that manner, then $De_{h}=e_{h-3}(x_{I^{c}})$ wherever $h\geq3$. We remark that we at this point we view these formulae as operations on polynomials, which produce polynomials and therefore we need not worry about the vanishing of denominators. The reader may, however, assume that we only consider points where all the variables take distinct values, and extend the results by continuity to the remaining points. This point becomes relevant in Corollary \ref{difcoord}, and is expanded on significantly in Section \ref{Diag} below.

In any case, we now wish to generalize these expressions.
\begin{lem}
Take a non-empty subset $I\subseteq\mathbb{N}_{N}$, and let $d>0$ be the size of $I$. Then the operator $D_{I}:=\sum_{i \in I}\partial_{i}\big/\prod_{i \neq j \in I}(x_{j}-x_{i})$ takes $e_{l}(x_{I})$ to 1 when $l=d$, and to 0 otherwise. \label{DIexI}
\end{lem}
Using the Kronecker $\delta$-symbol $\delta_{l,d}$, which equals 1 when $l=d$ and 0 otherwise, we can write the result of Lemma \ref{DIexI} as $D_{I}e_{l}(x_{I})=\delta_{l,d}$ for every $l$ and $I$ of size $d$. The analogue of $D_{I}$ in case $I$ is the empty set should not be the corresponding empty sum (which thus vanishes), but rather the identity operator, for it to take the function $e_{0}(x_{\emptyset})=1$ to 1, and not 0.

\begin{proof}
If $d=1$ then $I=\{i\}$, $D_{I}=\partial_{i}$ since the product is empty, and as we have $e_{0}(x_{I})=1$ and $e_{1}(x_{I})=x_{i}$, the result is clear. We therefore assume that $d\geq2$, where permuting the variables can be non-trivial. We then observe that the operator $D_{I}$ is invariant under applying a permutation to the elements of $I$, and is thus symmetric. Therefore if we let it act on a symmetric polynomial in the variables $\{x_{i}\}_{i \in I}$, and then multiply by the anti-symmetric expression $\prod_{i<j \in I}(x_{i}-x_{j})$ (according to the order of elements in $I$ as a subset of $\mathbb{N}_{N}$, say), then we obtain an anti-symmetric polynomial in these variables.

Now, an anti-symmetric polynomial in $\{x_{i}\}_{i \in I}$ vanishes wherever $x_{i}=x_{j}$, so that its zero set contains that of the irreducible polynomial $x_{i}-x_{j}$, making it divisible by $x_{i}-x_{j}$ in the UFD $\mathbb{C}[\{x_{i}\}_{i \in I}]$. Since this is valid for every pair $\{i,j\} \subseteq I$ with $i<j$, such a polynomial will be divisible by $\prod_{i<j \in I}(x_{i}-x_{j})$. In particular, if it does not vanish then its degree is at least the degree $\binom{d}{2}$ of the latter multiplier.

Consider now the action of $\prod_{i<j \in I}(x_{i}-x_{j}) \cdot D_{I}$ on the function $e_{l}(x_{I})$. Since the product contains every denominator from the definition of $D_{I}$, and every such denominator is homogenous of degree $d-1$, we deduce that this differential operator is the sum of $d$ operators, each of which is a partial derivative with respect to one variable multiplied by a homogenous polynomial of degree $\binom{d-1}{2}$. Letting such an operator act on $e_{l}$, of degree $l$, decreases the degree by 1 (via the derivatives), and then multiplies by a polynomial of degree $\binom{d-1}{2}$. The result is therefore an anti-symmetric polynomial of degree $\binom{d-1}{2}+l-1=\binom{d}{2}+l-d$. As this degree is smaller than $\binom{d}{2}$ when $l<d$, we deduce the asserted vanishing.

Finally, when $l=d$ the degree is $\binom{d}{2}$, meaning that $\prod_{i<j \in I}(x_{i}-x_{j}) \cdot D_{I}e_{d}$ is a constant multiple of $\prod_{i<j \in I}(x_{i}-x_{j})$. For evaluating the constant, which is the value of $D_{I}e_{d}$, assume that $I=\mathbb{N}_{d}\subseteq\mathbb{N}_{n}$ for easing the notation, and then in the anti-symmetric product $\prod_{i<j \in I}(x_{i}-x_{j})$, the monomial $\prod_{j=1}^{d}x_{j}^{d-j}$ shows up with a coefficient of 1. The constant value of $D_{I}e_{d}$ is therefore the coefficient with which this monomial shows up in $\prod_{i<j \in I}(x_{i}-x_{j}) \cdot D_{I}e_{d}$. But $x_{d}$ does not divide our monomial, and as $e_{d}=\prod_{j=1}^{d}x_{j}$, letting each $\partial_{i}$ with $i<d$ act on it and multiplying by some polynomial only gives combinations of monomials that are divisible by $x_{d}$. Thus our monomial only shows up in the summand associated with $i=d$ in $D_{I}$, where the derivative is $\prod_{j=1}^{d-1}x_{j}$ and the multiplying polynomial is $\prod_{i<j<d}(x_{i}-x_{j})$ after canceling. As our monomial appears with the coefficient 1 in the resulting product, the last value is also the desired one. This proves the lemma.
\end{proof}

We can now obtain the following consequence.
\begin{cor}
The operator $D_{I}$ from Lemma \ref{DIexI} sends the symmetric function $e_{h}=e_{h}(x_{\mathbb{N}_{N}})$ with $h \geq d$ to $e_{h-d}(x_{I^{c}})$. \label{ehDI}
\end{cor}

\begin{proof}
Write $e_{h}$ as in Lemma \ref{expeh}, and as $D_{I}$ contains no derivatives with respect to variables $x_{i}$ with $i \not\in I$, it only acts on the multipliers $e_{l}(x_{I})$ from that lemma. The result thus follows directly from Lemma \ref{DIexI}. This proves the corollary.
\end{proof}
Corollary \ref{ehDI} is also valid for $h<d$, if we define $e_{r}$ with negative $r$ to be 0. This is also visible in the examples calculated explicitly above. Moreover, the convention that $D_{I}=\operatorname{Id}$ (and not 0) when $I$ is empty is indeed the natural extension for this case, as this is the operator taking $e_{h}$ to $e_{h-0}(x_{I^{c}})$ for empty $I$. However, we shall not need this operator in what follows.

\smallskip

We are interested in symmetric differential operators in our variables $x_{i}$, $i\in\mathbb{N}_{N}$. The operator $D_{I}$ associated via Lemma \ref{DIexI} to a subset $I$ was seen to be symmetric, but only in the variables $\{x_{i}\}_{i \in I}$, and a permutation on $\{x_{i}\}_{i\in\mathbb{N}_{N}}$ that does not preserve $I$ clearly no longer leaves the operator $D_{I}$ invariant, but takes it to the operator of another set of size $d$. We shall use $|I|$ for the size of $I$ throughout this paper, and we consider the following symmetric operators and their action.
\begin{prop}
For every $d>0$, the operator $D_{d}:=d!\sum_{I\subseteq\mathbb{N}_{N},\ |I|=d}D_{I}$ is a symmetric differential operator, which takes $e_{h}$ with $h \geq d$ to $\frac{(N-h+d)!}{(N-h)!}e_{h-d}$ and $e_{h}$ with $h<d$ to 0. \label{Ddoneh}
\end{prop}

\begin{proof}
The symmetry of $D_{d}$ is obvious from the fact that permutations on $\{x_{i}\}_{i\in\mathbb{N}_{N}}$ permute the subsets of size $d$ in $\mathbb{N}_{N}$. Now, we saw in Corollary \ref{ehDI} that each $D_{I}$ takes $e_{h}$ to $e_{h-d}(x_{I^{c}})$, which equals $\sum_{H \leq I^{c},\ |H|=h-d}\prod_{i \in H}x_{i}$ in case $h \geq d$ and 0 otherwise. Thus $D_{d}e_{h}$ indeed vanishes when $h<d$, and if $h \geq d$ then in the image of $e_{h}$ under $D_{d}$, the expression $\prod_{i \in H}x_{i}$ for some subset $H\subseteq\mathbb{N}_{N}$ of size $h-d$ appears with one multiple of $d!$ for every subset $I\subseteq\mathbb{N}_{N}$ of size $d$ such that $H \subseteq I^{c}$. But as the latter condition is equivalent to $I \subseteq H^{c}$, it means that the number of such sets $I$ is obtained by choosing $d$ elements from the set $H^{c}$, of size $N-h+d$. As $d!$ times the number $\binom{N-h+d}{d}$ of such choices equals $\frac{(N-h+d)!}{(N-h)!}$ (independently of $H$, of course), and $\sum_{H\subseteq\mathbb{N}_{N},\ |H|=h-d}\prod_{i \in H}x_{i}$ is $e_{h-d}$ by definition, this proves the proposition.
\end{proof}
The reason for the factor $d!$ in the definition of $D_{d}$ in Proposition \ref{Ddoneh}, which can be considered as summing over $I$ as ordered $d$-tuples of distinct indices rather than subsets, will soon be seen to be more convenient. The convention of $e_{r}=0$ for $r<0$ makes the statement of Proposition \ref{Ddoneh} uniform for all $h$.

\smallskip

The coefficients from Proposition \ref{Ddoneh} depend on the number $N$ of variables. We thus renormalize each of our functions $e_{h}$ with $h\in\mathbb{N}\cup\{0\}$ by setting $\tilde{e}_{h}:=e_{h}\big/h!\binom{N}{h}=\frac{(N-h)!e_{h}}{N!}$, which have the property that when all the $x_{i}$'s take the same value $a$, this function attains $\frac{a^{h}}{h!}$, a value that no longer depends on $N$. Note that this indeed holds also for $h=0$, and we wrote the union $\mathbb{N}\cup\{0\}$ since it will be more convenient to later in this paper to only include positive integers in $\mathbb{N}$. We immediately obtain the following consequence.
\begin{cor}
The operator $D_{d}$ from Proposition \ref{Ddoneh} sends $\tilde{e}_{h}$ with $0 \leq h \leq N$ to $\tilde{e}_{h-d}$, which vanishes when $h<d$. \label{actnoreh}
\end{cor}

\begin{proof}
By Proposition \ref{Ddoneh}, $D_{d}$ takes $\tilde{e}_{h}=\frac{(N-h)!e_{h}}{N!}$ to $\frac{(N-h)!}{N!}\cdot\frac{(N-h+d)!}{(N-h)!}e_{h-d}$ (and to 0 when $h<d$), which indeed equals $\frac{(N-h+d)!e_{h-d}}{N!}=\tilde{e}_{h-d}$ as desired. This proves the corollary.
\end{proof}
Note that the condition $h \leq N$ is important in Corollary \ref{actnoreh}, since for an index $N<h \leq N+d$ there is no function $e_{h}$ that is mapped to $e_{h-d}$ via $D_{d}$. In Proposition \ref{Ddoneh} this issue did not show up, since for such $h$ the numerator $(N-h+d)!$ is finite (a factorial of a non-negative integer) while the denominator $(N-h)!$ is infinite as the factorial of a negative integer, meaning that $e_{h-d}$ is multiplied by 0 to give the correct value of $D_{d}e_{h}=D_{d}0$ for $e_{h}$ with $h>N$. Moreover, without the coefficient $d!$ in the definition of $D_{d}$, Corollary \ref{actnoreh} would have contained it, which is the reason for our normalization.

We can also normalize the products showing up in Theorem \ref{egen}, and set the expression $\tilde{e}_{\lambda}$ associated with a partition $\lambda \vdash r$ with $l(\lambda^{t}) \leq N$ to be the product $\prod_{q=1}^{l}\tilde{e}_{h_{q}}=\prod_{h=1}^{N}\tilde{e}_{h}^{m_{h}}$. It is a multiple of the original $\tilde{e}_{\lambda}$, which equals $a^{r}\big/\prod_{q=1}^{l}h_{q}!=a^{r}\big/\prod_{h=1}^{N}h!^{m_{h}}$ when $x_{i}=a$ for all $i\in\mathbb{N}_{N}$, and the spanning property and ring generation property from that theorem hold for these expressions as well, when we work over $\mathbb{Q}$. For describing the action of $D_{d}$ on these expressions we consider a partition $\lambda$ of $r$ and some $d \leq h \leq N$ that shows up in $\lambda$ (i.e., with $m_{h}\geq1$), and we define $\lambda-d\varepsilon_{h}$ to be the partition of $r-d$ that is obtained by replacing one instance of $h$ in $\lambda$ by $h-d$. We thus obtain the following result.
\begin{cor}
For every $1 \leq d \leq N$ and every partition $\lambda \vdash r$ with $l(\lambda^{t}) \leq N$, the operator $D_{d}$ takes $\tilde{e}_{\lambda}$ to $\sum_{h=d}^{N}m_{h}\tilde{e}_{\lambda-d\varepsilon_{h}}$. \label{difepart}
\end{cor}
In fact, since the maximal number $h$ for which $m_{h}>0$ is $l(\lambda^{t})$, we can take the sum to be only up to this value. In particular when $l(\lambda^{t})<d$, i.e., when all the multipliers constructing $\tilde{e}_{\lambda}$ are $\tilde{e}_{h}$ with $h<d$, Corollary \ref{difepart} yields $D_{d}\tilde{e}_{\lambda}=0$.

\begin{proof}
Since $D_{d}$ is a differential operator of order 1, its action on the product $\tilde{e}_{\lambda}$ is expressed via Leibniz' rule, which yields $\sum_{h=1}^{N}m_{h}\tilde{e}_{h}^{m_{h}-1}\prod_{g \neq h}\tilde{e}_{g}^{m_{g}} \cdot D_{d}\tilde{e}_{h}$. The summands with $h<d$ vanish, and if $h \geq d$ then the last multiplier is $\tilde{e}_{h}$ by Corollary \ref{actnoreh}. As for every $d \leq h \leq N$ this yields the multiplier $m_{h}$ times the product that defines $\tilde{e}_{\lambda-d\varepsilon_{h}}$, this proves the corollary.
\end{proof}
Note that Corollary \ref{difepart} is no longer valid for $d=0$ with the identity operator, since the latter operator does not satisfy the Leibniz' rule (and indeed, the asserted value of $D_{d}\tilde{e}_{\lambda}$ would become $\tilde{e}_{\lambda}$ times $\sum_{h=1}^{N}m_{h}=l(\lambda)$, rather than 1).

We remark that the if $I$ is a set of size $d$ as above then $\prod_{i \in I}\partial_{i}$ also takes each $e_{h}$ to $e_{h-d}(x_{I^{c}})$, as in Corollary \ref{ehDI}, and can thus be used for establishing counterparts for Proposition \ref{Ddoneh} and Corollary \ref{actnoreh}. However, the main difference shows up in Corollary \ref{difepart}, where the proof uses the fact that $D_{I}$ or $D_{d}$ are operators of degree 1 (at a generic point---the results in Section \ref{Diag} consider general points), which thus satisfies Leibniz' Rule, a property that is not shared by the product in question when $d\geq2$. For example, when $I=\{i,j\}$ of size 2, $D_{I}$ takes the product $e_{h}e_{g}$ to $e_{h-2}(x_{I^{c}})e_{g}+e_{h}e_{g-2}(x_{I^{c}})$, while the action of $\partial_{i}\partial_{j}$ will produce these two terms plus $e_{h-1}(x_{\{i\}^{c}})e_{g-1}(x_{\{j\}^{c}})+e_{h-1}(x_{\{j\}^{c}})e_{g-1}(x_{\{i\}^{c}})$. This is the reason why certain proofs along the total diagonal does not simplify in an obvious manner---see Remark \ref{nocomp} below.

\smallskip

We can now define our symmetric coordinates, and prove that they are the only ones with the property that we seek.
\begin{thm}
For every partition $\lambda \vdash r$ of some $1 \leq r \leq N$, let $m_{h}$ be the multiplicity with which the integer $1 \leq h \leq r$ appears in $\lambda$, and let $l(\lambda)$ be its length $\sum_{h=1}^{r}m_{h}$. We then set \[u_{r}:=\sum_{\lambda \vdash r}(-1)^{l(\lambda)-1}\frac{\big(l(\lambda)-1\big)!}{\prod_{h=1}^{r}m_{h}!}\tilde{e}_{\lambda},\] where $\tilde{e}_{\lambda}$ is the normalized product appearing in Corollary \ref{difepart}. Then $u_{r}$ is the only symmetric polynomial in $x_{i}$, $i\in\mathbb{N}_{N}$ that is homogenous of degree $r$ and satisfies the property that $D_{d}u_{r}$ is the Kronecker $\delta$-symbol $\delta_{d,r}$. \label{coorddef}
\end{thm}

\begin{rmk}
Recall that the partial ordinary Bell polynomial $\hat{B}_{r,t}$ of some variables $z_{h}$, $1 \leq h \leq r-t+1$ can be expressed as $\sum_{\lambda \vdash r,\ l(\lambda)=t}t!\prod_{h=1}^{r-t+1}z_{h}^{m_{h}}/m_{h}!$, and it is homogenous of degree $t$ when all the variables have degree 1 (when $z_{h}$ has degree $h$, the homogeneity degree is $r$), so that inverting the signs of the variables multiplies $\hat{B}_{r,t}$ by $(-1)^{t}$. As $\tilde{e}_{\lambda}$ is the product $\prod_{h=1}^{r}\tilde{e}_{h}^{m_{h}}$, we deduce that the coordinate $u_{r}$ from Theorem \ref{coorddef} can be written in short as $-\sum_{t=1}^{r}\hat{B}_{r,t}(-\tilde{e}_{1},\ldots,-\tilde{e}_{r-t+1})/t$. A very similar observation will be useful for Proposition \ref{Ntoinfty} below. \label{Bell}
\end{rmk}

\begin{proof}[Proof of Theorem \ref{coorddef}]
We begin with the uniqueness of $u_{r}$, recalling that the homogeneity of degree $r$ implies, via Theorem \ref{egen}, that $u_{r}$ has a unique presentation as $\sum_{\lambda \vdash r}c_{\lambda}\tilde{e}_{\lambda}$ (as $r \leq N$, the condition $l(\lambda^{t}) \leq N$ is redundant for partitions $\lambda \vdash r$). Moreover, since when $d>r$ we have $m_{h}=0$ for every $h \geq d$ in every partition $\lambda \vdash r$, the condition $D_{d}u_{r}=0$ holds for such $d$ regardless of the choice of the homogenous expression $u_{r}$ of degree $r$.

Now, when $d=r$ the partition $\lambda_{r}$ of $r$ of length 1 contains $\tilde{e}_{r}$ (in fact, $\tilde{e}_{\lambda_{r}}=\tilde{e}_{r}$), and all the other partitions are only based on integers $1 \leq h<r$. It follows from Corollary \ref{difepart} (and even Corollary \ref{actnoreh}) that if $u_{r}$ is presented as above then $D_{r}u_{r}$ just gives the coefficient $c_{\lambda_{r}}$. Thus the equality $D_{r}u_{r}=1$ implies (and is equivalent to) $c_{\lambda_{r}}=1$.

We now prove the uniqueness of the value of $c_{\lambda}$ by decreasing induction on the maximal number showing up in $\lambda$ (namely on $l(\lambda^{t})$ as before). We have established the basis of the induction with $l(\lambda^{t})=r$ above, so assume that we have determined the value of $c_{\lambda}$ for every $\lambda \vdash r$ with $l(\lambda^{t})>d$ for some $d<r$ by the equalities $D_{p}u_{r}=\delta_{p,r}$ for every $p>d$, and we wish to find the values of $c_{\lambda}$ when $l(\lambda^{t})=d$. Note that $D_{p}$ with $p>d$ annihilates $\tilde{e}_{\lambda}$ as above, so that this coefficient did not show up in the previous calculations, and consider the equality $D_{d}u_{r}=0$. The expressions $\tilde{e}_{\lambda}$ with $l(\lambda^{t})<d$ are annihilated by $D_{d}$ once again, and as $\lambda-d\varepsilon_{h}$ is a partition of $r-d$ for every $d \leq h \leq r$, we deduce that the image under $D_{d}$ of the determined combination $\sum_{\lambda \vdash r,\ l(\lambda^{t})>d}c_{\lambda}\tilde{e}_{\lambda}$ is some combination $\sum_{\rho \vdash r-d}b_{\rho}\tilde{e}_{\rho}$. Now, for every $\lambda \vdash r$ with $l(\lambda^{t})=d$ Corollary \ref{difepart} (with the precise bound $l(\lambda^{t})$ on $h$) yields $D_{d}\tilde{e}_{\lambda}=m_{d}\tilde{e}_{\lambda-d\varepsilon_{d}}$, and the map $\lambda\mapsto\lambda-d\varepsilon_{d}$ is an injective map from the partitions $\lambda \vdash r$ with $l(\lambda^{t})=d$ into the partitions $\rho \vdash r-d$ (it is a bijection when we impose the condition $l(\rho^{t}) \leq d$ on the latter partitions). Therefore the only choice for obtaining $D_{d}u_{r}=0$ is by choosing the coefficient $c_{\lambda}$ for such $\lambda$ to be $-b_{\lambda-d\varepsilon_{d}}/m_{d}$, where $b_{\lambda-d\varepsilon_{d}}$ is given in terms of the values that we had already determined by the induction hypothesis. This means that if an expression homogenous symmetric $u_{r}$ that satisfies $D_{d}u_{r}=\delta_{d,r}$ for every $d$ exists, then it is unique.

For proving the existence, we simply verify that our explicit formula for $u_{r}$ has the desired property. We have already seen that $D_{d}u_{r}=0$ when $d>r$ in any case, and since the coefficient associated with $\lambda_{r}$, having length 1 and only one positive multiplicity $m_{r}=1$, equals 1, we get $D_{r}u_{r}=1$ as before. So assume that $d<r$, and we have to verify that $D_{d}u_{r}=0$. Using Corollary \ref{difepart} (and linearity) we write $D_{d}u_{r}$ as \[D_{d}\sum_{\lambda \vdash r}(-1)^{l(\lambda)-1}\frac{\big(l(\lambda)-1\big)!}{\prod_{h=1}^{r}m_{h}!}\tilde{e}_{\lambda}=\sum_{\lambda \vdash r}(-1)^{l(\lambda)-1}\frac{\big(l(\lambda)-1\big)!}{\prod_{h=1}^{r}m_{h}!}\cdot\sum_{h=d}^{N}m_{h}\tilde{e}_{\lambda-d\varepsilon_{h}},\] and for every $\rho \vdash r-d$, in which every $h$ appears with multiplicity $\mu_{h}$, we have to prove that the total contribution to the multiplier of $\tilde{e}_{\rho}$ vanishes.

Now, the terms yielding a multiple of $\tilde{e}_{\rho}$ are the ones coming from $\lambda \vdash r$ and $h \geq d$ such that $\lambda-d\varepsilon_{h}=\rho$. When $h>d$ this means that one instance of $h$ was replaced by $h-d$ when we moved from $\lambda$ to $\rho$, meaning that we have $m_{h}=\mu_{h}+1$, $m_{h-d}=\mu_{h-d}-1$, and $m_{g}=\mu_{g}$ for every other value of $g$. Thus $l(\lambda)=\sum_{h=1}^{r}m_{h}=\sum_{h=1}^{r}\mu_{h}=l(\rho)$, and the additional multiplier is \[\frac{m_{h}}{\prod_{g=1}^{r}m_{g}!}=\frac{m_{h-d}+1}{(m_{h}-1)!(m_{h-d}+1)!\prod_{g \neq h,h-d}m_{g}!}=\frac{\mu_{h-d}}{\prod_{g=1}^{r-d}\mu_{g}!}\] (we can let $g$ go only up to $r-d$ since $\rho \vdash r-d$). This means that the contribution of the term arising from such $h$ to the coefficient multiplying expression $\tilde{e}_{\rho}$ is $(-1)^{l(\rho)-1}\big(l(\rho)-1\big)!\big/\prod_{g=1}^{r-d}\mu_{g}!$ times $\mu_{h-d}$.

However, we also have the combination with $h=d$ and $\lambda$ with $\lambda-d\varepsilon_{d}=\rho$, for which going from $\lambda$ to $\rho$ simply means discarding one instance of $d$. Therefore $m_{d}=\mu_{d}+1$ and $m_{g}=\mu_{g}$ for $g \neq d$, so that $l(\lambda)=l(\rho)+1$ and $m_{d}\big/\prod_{h=1}^{r}m_{h}!$ reduces to $1\big/\prod_{g=1}^{r-d}\mu_{g}!$, and as $(-1)^{l(\lambda)-1}\big(l(\lambda)-1\big)!=-(-1)^{l(\rho)-1}\big(l(\rho)\big)!$, we find that this is the coefficient $(-1)^{l(\rho)-1}\big(l(\rho)-1\big)!\big/\prod_{g=1}^{r-d}\mu_{g}!$ times $\tilde{e}_{\rho}$ multiplied by $-l(\rho)$. Gathering all the multiples of $(-1)^{l(\rho)-1}\big(l(\rho)-1\big)!\big/\prod_{g=1}^{r-d}\mu_{g}!$ times $\tilde{e}_{\rho}$ yields the total value \[\textstyle{\sum_{h=d+1}^{r}\mu_{h-d}-l(\rho)=\sum_{g=1}^{d-r}\mu_{g}-l(\rho)=0},\] as desired. Therefore our $u_{r}$ satisfies the desired equalities, and it was seen above to be unique once it exists. This proves the theorem.
\end{proof}
Note that the proof of the uniqueness in Theorem \ref{coorddef} was not enough to establish the existence of $u_{r}$, since at the induction step associated with $d$ we could only make sure to cancel the multiple of $\tilde{e}_{\rho}$ for $\rho \vdash r-d$ with $l(\rho^{t}) \leq d$ by choosing the (unique) appropriate values for $c_{\lambda}$ for $\lambda \vdash r$ with $l(\lambda^{t})=d$. That argument alone did not suffice to verify that if $l(\rho^{t})>d$ then the coefficient of $\tilde{e}_{\rho}$ in that $D_{d}$-image, which cannot be affected by the $c_{\lambda}$'s in question, indeed vanishes.

\smallskip

The fact that the expression $u_{r}$ from Theorem \ref{coorddef} is $\tilde{e}_{r}$ plus a polynomial expression in $\{\tilde{e}_{h}\}_{h=1}^{r-1}$, with rational coefficients, implies, via Theorem \ref{egen}, that $\{u_{r}\}_{r=1}^{N}$ also generate the ring of symmetric functions in $\{x_{i}\}_{i\in\mathbb{N}_{N}}$ over $\mathbb{Q}$. We can therefore use them as coordinates for presenting any symmetric function $\phi$ of $\{x_{i}\}_{i\in\mathbb{N}_{N}}$, and with this presentation the respective derivatives are only based on derivatives of $\phi$ and denominators, with no multipliers.
\begin{cor}
Let $\phi$ be a continuously differentiable symmetric function of $N$ real variables. If we write $\phi(x_{1},\ldots,x_{N})$ as $\varphi(u_{1},\ldots,u_{N})$ for some function $\varphi$, with the $u_{r}$'s being those from Theorem \ref{coorddef}, then we have the equality $\varphi_{u_{d}}=D_{d}\phi$ for every $1 \leq d \leq N$, holding at every point in which the values of the $x_{i}$'s are all distinct. \label{difcoord}
\end{cor}

\begin{proof}
For every $1 \leq d \leq N$, the operator $D_{d}$ is a differential operator of degree 1, which can be applied to any differentiable function at any point $(x_{1},\ldots,x_{N})$ for which $x_{i} \neq x_{j}$ wherever $i \neq j$ (this condition is required for the denominators in the definition of $D_{d}$ not to vanish). Therefore when we apply it on $\phi=\phi(x_{1},\ldots,x_{N})=\varphi(u_{1},\ldots,u_{N})$, we get, by the chain rule, the expression $\sum_{r=1}^{N}D_{d}u_{r}\cdot\varphi_{u_{r}}$. As Theorem \ref{coorddef} evaluates $D_{d}u_{r}$ as $\delta_{d,r}$, the latter sum reduces to $\varphi_{u_{d}}$, as desired. This proves the corollary.
\end{proof}
The result of Corollary \ref{difcoord} is equivalently valid for holomorphic symmetric functions in $N$ complex variables. So will all the results about derivatives below.

Another consequence that we obtain from Theorem \ref{coorddef} is the following one.
\begin{cor}
The Weyl algebra of symmetric functions in $N$ variables is generated by $\{u_{r}\}_{r=1}^{N}$ (all commuting with one another) and $\{D_{d}\}_{d=1}^{N}$ (commuting with one another as well), with the commutation relation $[D_{d},u_{r}]=\delta_{d,r}$. \label{Weyl}
\end{cor}

\begin{proof}
We already saw that Theorem \ref{egen} allows us to present the ring of symmetric polynomials in $N$ variables as the polynomial algebra generated by $\{u_{r}\}_{r=1}^{N}$. Since Theorem \ref{coorddef} shows that $\{D_{d}\}_{d=1}^{N}$ is the dual system of derivations on that algebra, joining them to this algebra produces the corresponding Weyl algebra (and in particular the commutation of the $D_{d}$'s with one another). This proves the corollary.
\end{proof}

\section{Properties of the Coordinates $\{u_{r}\}_{r=1}^{N}$ \label{PropCoord}}

In this section we establish some of the properties of coordinates $\{u_{r}\}_{r=1}^{N}$ from Theorem \ref{coorddef}, and see alternative ways to characterize them.

\smallskip

First we state that the homogeneity and the formulae from Corollary \ref{difcoord} essentially characterize the coordinates $\{u_{r}\}_{r=1}^{N}$ from Theorem \ref{coorddef}.
\begin{thm}
Take, for every $1 \leq r \leq N$, a homogenous symmetric function $v_{r}$ of degree $r$ in $\{x_{i}\}_{i\in\mathbb{N}_{N}}$. Assume that $\{v_{r}\}_{r=1}^{N}$ are coordinates for symmetric functions, such that if we write a symmetric function $\phi=\phi(x_{1},\ldots,x_{N})$ as $\eta(v_{1},\ldots,v_{N})$, then for every $r$ the derivative $\eta_{v_{r}}$ involves derivatives with respect to the $x_{i}$'s and denominators, with no numerators, and with the terms from $D_{r}\phi$ showing up multiplied by 1. Then $v_{r}=u_{r}$ for every $1 \leq r \leq N$. \label{uniqders}
\end{thm}
While one can prove Theorem \ref{uniqders} now already, we postpone the proof until we have some formulae from the next section, which will make the proof significantly simpler. The reader can verify that this theorem is not used anywhere in the following section, so that there is no circular reasoning in proving it in this way.

\smallskip

Recall that a diagonal point in a 2-dimensional space is a point where the values of the two variables are the same. Here we have, in general, more than two variables, so the only symmetric set that we can define in this manner is the \emph{total diagonal}, consisting of the points in $\mathbb{R}^{N}$ at which all the variables $\{x_{i}\}_{i\in\mathbb{N}_{N}}$ take the same value. An argument generalizing the proof of Theorem \ref{coorddef} yields the following additional property of our coordinates.
\begin{thm}
Take any simple derivative $\prod_{q=1}^{d}\partial_{i_{q}}$ of order $d$ (where some of the indices $i_{q}$, $1 \leq q \leq d$ are allowed to coincide). Then when we apply it to $u_{r}$ with $r \neq d$, the resulting function vanishes along the total diagonal. \label{annbyders}
\end{thm}

For proving Theorem \ref{annbyders}, we first make the following evaluation, of the derivative from that theorem on an expression $\tilde{e}_{\lambda}$ as in Corollary \ref{difepart}.
\begin{lem}
Let $\lambda$ be a partition of some integer $r$, in which each number $1 \leq h \leq r$ appears with multiplicity $m_{h}$, and consider the operator $\prod_{q=1}^{d}\partial_{i_{q}}$ from Theorem \ref{annbyders}. If $\iota:\mathbb{N}_{d}\to\mathbb{N}_{N}$ is the map taking $1 \leq q \leq d$ to the index $i_{q}$, then denote by $X_{\iota}$ the set of all partitions of $\mathbb{N}_{d}$ into a disjoint union of sets $\{J_{\nu}\}_{\nu=1}^{\gamma}$, such that the indices $i_{q}$ for $q$ in the same set $J_{\nu}$ are all distinct. Then the value of $\partial^{d}\tilde{e}_{\lambda}\big/\prod_{q=1}^{d}\partial_{i_{q}}$ is the sum over all elements of $X_{\iota}$ of the expression \[\sum_{\substack{h_{\nu}\geq1,\ 1\leq\nu\leq\gamma \\ \sum_{\nu=1}^{\gamma}h_{\nu} \leq r}}\prod_{\nu=1}^{\gamma}\big(m_{h_{\nu}}-\textstyle{\sum_{\sigma=1}^{\nu-1}\delta_{h_{\nu},h_{\sigma}}}\big)\displaystyle \tilde{e}_{\lambda-\sum_{\nu=1}^{\gamma}h_{\nu}\varepsilon_{h_{\nu}}}\prod_{\nu=1}^{\gamma}\frac{\tilde{e}_{h_{\nu}-\kappa_{\nu}}(\{i_{q}|q \in J_{\nu}\}^{c})}{N!/(N-\kappa_{\nu})!},\] where $\kappa_{\nu}$ is the size of the set $J_{\nu}$ in the given element of $X_{\iota}$ for any $1\leq\nu\leq\gamma$, which hence satisfy $\sum_{\nu=1}^{\gamma}\kappa_{\nu}=d$. \label{derdelambda}
\end{lem}

\begin{proof}
We argue by induction on $d$. For $d=1$ the derivative $\partial_{i}$ is simply the operator $D_{I}$ for the singleton $I=\{i\}$, where Corollary \ref{actnoreh} gives $\partial_{i}e_{h}=e_{h-1}(\{i\}^{c})$. Recalling the normalization from Corollary \ref{actnoreh}, of $\tilde{e}_{h}$ as $\frac{(N-h)!e_{h}}{N!}$ and thus of $\tilde{e}_{h-1}(\{i\}^{c})$ as $\frac{(N-h)!e_{h-1}(\{i\}^{c})}{(N-1)!}$, we deduce that $\partial_{i}\tilde{e}_{h}=\frac{\tilde{e}_{h-1}(\{i\}^{c})}{N}$. It follows, via Leibniz' Rule, that for our $\lambda \vdash r$ we have $\partial_{i}\tilde{e}_{\lambda}=\sum_{h=1}^{r}m_{h}\tilde{e}_{\lambda-h\varepsilon_{h}}\frac{\tilde{e}_{h-1}(\{i\}^{c})}{N}$ in the notation from the proof of Theorem \ref{coorddef}. This proves the case with $d=1$.

For higher values of $d$, note that our expression for $\partial\tilde{e}_{\lambda}/\partial_{i}$ contains $\tilde{e}_{\mu}$'s for partitions $\mu$, but also $\tilde{e}_{\nu}(\{i\}^{c})$ for partitions $\nu$ (in this case of length 1). When we apply another operator $\partial_{j}$, its action on $\tilde{e}_{\lambda-h\varepsilon_{h}}$ from the $g$th summand will produce $\sum_{g=1}^{r-h}(m_{g}-\delta_{g,h})\tilde{e}_{\lambda-h\varepsilon_{h}-g\varepsilon_{g}}\frac{\tilde{e}_{g-1}(\{j\}^{c})}{N}$ (the expression with the $\delta$-symbol showing up because the multiplicity of $h$ in $\lambda-h\varepsilon_{h}$ is the same $m_{g}$ when $g \neq h$, but equals $m_{h}-1$ if $g=h$ since we took out one instance of $h$), by what we just saw. However, it can also act on the multiplier $\frac{\tilde{e}_{h-1}(\{i\}^{c})}{N}$, which gives $\frac{\tilde{e}_{h-2}(\{i,j\}^{c})}{N(N-1)}$ if $j \neq i$ by the same argument, and just 0 in case $j=i$ since that expression does not depend on $x_{i}$ anymore. Combining these two parts we obtain the result for $d=2$, since the set of $X_{\iota}$ contains the partition of $\mathbb{N}_{2}$ into two singletons, as well as the partition into one set if $i \neq j$.

Continuing by induction, assume now that $\partial^{d}\tilde{e}_{\lambda}\big/\prod_{q=1}^{d}\partial_{i_{q}}$ is expressed via the induction hypothesis, apply another operator $\partial_{i_{d+1}}$, to obtain the image of $\tilde{e}_{\lambda}$ under a general such operator of degree $d+1$. Let $\tilde{\iota}$ be the extension of $\iota$ to $\mathbb{N}_{d+1}$ taking $d+1$ to $i_{d+1}\in\mathbb{N}_{N}$. Take, inside the summand associated with the element of $X_{\iota}$ containing the sets $\{J_{\nu}\}_{\nu=1}^{\gamma}$ (of respective cardinalities $\{\kappa_{\nu}\}_{\nu=1}^{\gamma}$), the term with the values $\{h_{\nu}\}_{\nu=1}^{\gamma}$, and let $\partial_{i_{d+1}}$ acts on it.

Now, the partition $\lambda-\sum_{\nu=1}^{\gamma}h_{\nu}\varepsilon_{h_{\nu}}$ is of $r-\sum_{\nu=1}^{\gamma}h_{\nu}$, and it contains each possible value of a new index $1 \leq h_{\gamma+1} \leq r-\sum_{\nu=1}^{\gamma}h_{\nu}$ with the multiplicity $m_{h_{\gamma+1}}-\sum_{\sigma=1}^{\gamma}\delta_{h_{\gamma+1},h_{\sigma}}\big)$. Therefore $\partial_{i_{d+1}}$ takes $\tilde{e}_{\lambda-\sum_{\nu=1}^{\gamma}h_{\nu}\varepsilon_{h_{\nu}}}$, by what we saw above, to the sum over such $h_{\gamma+1}$ of this respective multiplicity times $\tilde{e}_{\lambda-\sum_{\nu=1}^{\gamma+1}h_{\nu}\varepsilon_{h_{\nu}}}\frac{\tilde{e}_{h_{\gamma+1}-1}(\{i_{d+1}\}^{c})}{N}$ (when this partition is empty, the vanishing of this derivative will be expressed in terms of the vanishing multiplicities). Multiplying by the remaining expressions and summing over all the values of $\{h_{\nu}\}_{\nu=1}^{\gamma}$ thus produces the sum corresponding to the element of $X_{\tilde{\iota}}$ obtained from our element of $X_{\iota}$ by adding the singleton $J_{\gamma+1}=\{d+1\}$.

But the operator $\partial_{i_{d+1}}$ can also act on the multiplier $\tilde{e}_{h_{\nu}-\kappa_{\nu}}(\{i_{q}|q \in J_{\nu}\}^{c})$ for some $1\leq\nu\leq\gamma$. If $i_{d+1}$ coincides with $i_{q}$ for some $q \in J_{\nu}$, then this function is independent of $x_{i_{d+1}}$ and the derivative is 0. When this is not the case, the derivative gives $\tilde{e}_{h_{\nu}-\kappa_{\nu}-1}(\{i_{q}|q \in J_{\nu}\}^{c} \setminus i_{d+1})/(N-\kappa_{\nu})$ as we saw above, and when we gather all of the terms obtained in this way, over every choice of $\{h_{\nu}\}_{\nu=1}^{\gamma}$, this yields the sum associated with the element of $X_{\tilde{\iota}}$ in which $d+1$ is added to the set $J_{\nu}$, thus keeping the condition that all the indices $i_{q}$ for $q$ in this set are distinct (and now it is not a singleton), and the other sets remain the same. When we combine everything together, and sum over all elements of $X_{\iota}$, we indeed obtain all the sums arising from all the elements of $X_{\tilde{\iota}}$. This proves the lemma.
\end{proof}

\begin{proof}[Proof of Theorem \ref{annbyders}]
The result is obvious for $d>r$ by the homogeneity degree of $u_{r}$, so assume $d<r$. We express $u_{r}$ as in Theorem \ref{coorddef}, and then for every $\lambda \vdash r$ the image of $\tilde{e}_{\lambda}$ showing up in $u_{r}$ under our operator is given by Lemma \ref{derdelambda}. At a point on the total diagonal, say where $x_{i}=a$ for every $i\in\mathbb{N}_{N}$, each term $\tilde{e}_{g}(I)$, of every set $I$, attains the value $\frac{a^{g}}{g!}$, so that the value of the product of $\tilde{e}_{\lambda-\sum_{\nu=1}^{\gamma}h_{\nu}\varepsilon_{h_{\nu}}}$ and of $\tilde{e}_{h_{\nu}-\kappa_{\nu}}(\{i_{q}|q \in J_{\nu}\}^{c})$ over $1\leq\nu\leq\gamma$ from that lemma attains at our point the same value as that of $\tilde{e}_{\lambda-\sum_{\nu=1}^{\gamma}\kappa_{\nu}\varepsilon_{h_{\nu}}}$. Note that since $\lambda \vdash r$ and $\sum_{\nu=1}^{\gamma}\kappa_{\nu}=d$, this expression corresponds to a partition of $r-d>0$. We thus fix one element of $X_{\iota}$ from Lemma \ref{derdelambda} (and with it the $J_{\nu}$'s and their cardinalities, the $\kappa_{\nu}$'s), multiply the expression arising from $\lambda$ by the coefficient $(-1)^{l(\lambda)-1}\big(l(\lambda)-1\big)!\big/\prod_{h=1}^{r}m_{h}!$ from Theorem \ref{coorddef}, sum over $\lambda$, fix a partition $\rho \vdash r-d$, with multiplicities $\{\mu_{h}\}_{h=1}^{r-d}$, and consider the resulting coefficient of the value of $\tilde{e}_{\rho}\big/\prod_{\nu=1}^{\gamma}\frac{N!}{(N-\kappa_{\nu})!}$ at our diagonal point.

Now, for fixed $\lambda$ and $\{h_{\nu}\}_{\nu=1}^{\gamma}$ such that $\lambda-\sum_{\nu=1}^{\gamma}\kappa_{\nu}\varepsilon_{h_{\nu}}=\rho$, denote by $\tilde{\mu}_{h}$ the multiplicity of $h$ in the partition $\lambda-\sum_{\nu=1}^{\gamma}h_{\nu}\varepsilon_{h_{\nu}}$. Then the product of the expressions $m_{h_{\nu}}-\sum_{\mu=1}^{\nu-1}\delta_{h_{\nu},h_{\mu}}$ cancels with the denominator $\prod_{h=1}^{r}m_{h}!$ to give a denominator of $\prod_{h}\tilde{\mu}_{h}!$. We also note that for each $\nu$ such that $\kappa_{\nu}=h_{\nu}$ we obtain a difference of 1 between $l(\rho)$ and $l(\lambda)$, and when $\kappa_{\nu}<h_{\nu}$ we obtain a difference of 1 between $\mu_{g_{\nu}}$ and $\tilde{\mu}_{g_{\nu}}$, where $g_{\nu}=h_{\nu}-\kappa_{\nu}$. Letting $\lambda$ and the $h_{\nu}$'s vary (but with the equality $\lambda-\sum_{\nu=1}^{\gamma}\kappa_{\nu}\varepsilon_{h_{\nu}}=\rho$ holding), for every $\nu$ we either run over all the values of $g_{\nu}$ (thus increasing the multiplicity of $g_{\nu}$, giving a multiplier of $\mu_{g_{\nu}}$ minus the $\delta$-symbols with the previous ones in the comparison between $\prod_{h}\tilde{\mu}_{h}!$ and $\prod_{h}\mu_{h}!$ in the numerator), or add 1 to $l(\rho)$ (and gives a multiplier in the comparison between $l(\rho)$ and the length $l(\lambda)-\gamma$ of $\lambda-\sum_{\nu=1}^{\gamma}h_{\nu}\varepsilon_{h_{\nu}}$).

Therefore, by taking out a multiplier of $(-1)^{l(\rho)-1}\big(l(\rho)-1\big)!\big/\prod_{h=1}^{r}\mu_{h}!$ (which we can do since $r-d>0$ and $\rho$ is non-trivial), we see that if $\beta$ of the indices $1\leq\nu\leq\gamma$ give numbers $g_{\nu}$ and the other ones increase the length, then the contribution from such $\lambda$'s combine to \[(-1)^{\beta}\binom{\gamma}{\beta}\prod_{t=0}^{\gamma-\beta-1}\big(l(\rho)+t\big)\cdot \sum_{g_{1}\geq1}\ldots\sum_{g_{\beta}\geq1}\prod_{\nu=1}^{\beta}\big(\mu_{g_{\nu}}-\textstyle{\sum_{\sigma=1}^{\nu-1}\delta_{g_{\nu},g_{\sigma}}}\big),\] where the binomial coefficient represents the choice of which indices increase the length. Now, the sum over $g_{\beta}$ gives $l(\rho)-\beta+1$ for every choice of the previous ones, then the sum over $g_{\beta-1}$ yields $l(\rho)-\beta+2$, and so forth, the total sum divided by $\beta!$ from the binomial coefficient is $\binom{l(\rho)}{\beta}$. Similarly, the product over $t$ and the sign combine with the denominator $(\gamma-\beta)!$ to the extended binomial coefficient $\binom{-l(\rho)}{\gamma-\beta}$. By summing over $\beta$ and recalling that $\sum_{\beta=0}^{\gamma}\binom{z}{\beta}\binom{w}{\gamma-\beta}$ equals $\binom{z+w}{\gamma}$ as extended binomial coefficients for every $z$ and $w$, we obtain $\gamma!\binom{0}{\gamma}$, which vanishes since we always have $\gamma\geq1$ when $d\geq1$. This shows that the combination coming from $\rho$ comes multiplied by a vanishing coefficient for every $\rho$, so that $\partial^{d}u_{r}\big/\prod_{q=1}^{d}\partial_{i_{q}}$ indeed vanishes at our point on the total diagonal. This proves the theorem.
\end{proof}

\begin{cor}
consider the point on the total diagonal where the common value of all the variables $x_{i}$, $i\in\mathbb{N}_{N}$ is $a$. Then at this point we have $u_{1}=a$ and $u_{r}=0$ for every $r\geq2$. The total diagonal is characterized algebraically in the $N$th symmetric power by the vanishing of $\{u_{r}\}_{r=2}^{N}$, over any field of characteristic 0. \label{coordiag}
\end{cor}

\begin{proof}
As $u_{r}$ is homogenous of degree $r$, we have the equality $\sum_{i=1}^{N}x_{i}\partial_{i}u_{r}=ru_{r}$ as functions on $\mathbb{R}^{N}$. Now, Theorem \ref{annbyders} shows that if $r\geq2$ then $\partial_{i}u_{r}=0$ for all $i$ at our point, meaning that the left hand side vanishes at that point hence so does the right hand side, i.e., so does $u_{r}$. For $r=1$ both sides reduce to $\sum_{i=1}^{N}x_{i}/N$ (indeed, $u_{1}=\tilde{e}_{1}=\frac{e_{1}}{N}$, so that $\partial_{i}u_{1}=\frac{1}{N}$ as we saw in, e.g., the proof of Theorem \ref{annbyders}), which clearly attains the value $a$ at our point. The second assertion now follows from the first via the fact that $\{u_{r}\}_{r=1}^{N}$ are independent generators for the ring of symmetric polynomials over any such field, as we saw from Theorem \ref{egen}. This proves the corollary.
\end{proof}

\begin{rmk}
Neither the property from Corollary \ref{coordiag} nor that from Theorem \ref{annbyders} characterizes $\{u_{r}\}_{r=1}^{N}$, even up to scalar multiples. For seeing this we define, for a partition $\eta \vdash r$ with multiplicities $\{\mu_{h}\}_{h=1}^{r}$, the expression $u_{\eta}:=\prod_{h=1}^{r}u_{h}^{\mu_{h}}$ as usual. It is clear that $u_{\eta}$ is homogenous of degree $r$, and it follows from Corollary \ref{coordiag} that if $\mu_{h}\geq1$ for some $h\geq2$ (i.e., when $\eta$ is not the partition of length $r$ of $r$) then $u_{\eta}$ vanishes along the total diagonal. It follows that if $r\geq3$ and $\eta$ is any partition with $2 \leq l(\eta) \leq r-1$ then we can change $u_{r}$ by any multiple of $u_{\eta}$ and still get a set of coordinates with the property from that Corollary \ref{coordiag}. As for the derivatives as in Theorem \ref{annbyders}, take any $u_{\eta}$ for some $\eta \vdash r$ with $l(\eta)\geq2$ that does not contain the value 1 (the smallest such partition is that of 4 as $2+2$), and consider an operator of the form from Theorem \ref{annbyders} that has some degree $d<r$. As applying it to $u_{\eta}$ will give linear combinations of $\{u_{r}\}_{r=2}^{N}$ and their derivatives, all of which vanish along the total diagonal (by Theorem \ref{annbyders} and Corollary \ref{coordiag}), we deduce that such $u_{\eta}$ also has the property from Theorem \ref{annbyders}. A characterization of our coordinates $\{u_{r}\}_{r=1}^{N}$ via derivatives (along the total diagonal and elsewhere) is presented in Theorem \ref{ordwithN} below. \label{charorddiag}
\end{rmk}

\smallskip

Next we consider the dependence of our coordinates on $N$, and the variation with $N$. Recall that the classical definition of the ring of symmetric functions (say over $\mathbb{Q}$), in infinitely many variables, is given in \cite{[M]} and others as the direct sum over $d$ of the inverse limit of the spaces $\mathbb{Q}[x_{1},\ldots,x_{N}]_{d}^{S_{N}}$ of symmetric functions in $N$ variables that are homogenous of degree $d$, where the maps in the inverse limit construction are the ones going from $\mathbb{Q}[x_{1},\ldots,x_{N+1}]^{S_{N+1}}$ to $\mathbb{Q}[x_{1},\ldots,x_{N}]^{S_{N}}$ (as graded rings) and taking $x_{i}$ with $i \leq N$ to itself and $x_{n+1}$ to 0. We denote this ring, which is also graded by the degrees, by $\Lambda_{\mathbb{Q}}$ (since it is the extension of scalars from $\mathbb{Z}$ to $\mathbb{Q}$ from the ring denoted by $\Lambda$ in \cite{[M]}).

Now, for a fixed degree $d$, the maps used in the inverse limit construction for the $d$-homogenous part of $\Lambda_{\mathbb{Q}}$ are isomorphisms for every $N \geq d$. Therefore as maps of additive groups they can be inverted, thus producing the $d$-homogenous part of $\Lambda_{\mathbb{Q}}$ (and, in fact, all of $\Lambda_{\mathbb{Q}}$ at once) as a direct limit as well. The corresponding map from $\mathbb{Q}[x_{1},\ldots,x_{N}]_{d}^{S_{N}}$ to $\mathbb{Q}[x_{1},\ldots,x_{N+1}]_{d}^{S_{N+1}}$ takes the basic symmetric monomial $m_{\eta}(x_{1},\ldots,x_{N})$ associated with some partition $\eta$ of $d$ to $\frac{1}{N+1-l(\eta)}\sum_{\sigma \in S_{N+1}/S_{N}}\sigma(m_{\eta})$, as an element of $\mathbb{Q}[x_{1},\ldots,x_{N+1}]_{d}^{S_{N+1}}$. Indeed, the term associated with every $\sigma \in S_{N+1}/S_{N}$ is well-defined (because $m_{\eta}(x_{1},\ldots,x_{N})$ is already $S_{N}$-invariant), the sum is clearly $S_{N+1}$-invariant and is based only on monomials from $m_{\eta}$, and the multiplying coefficient is one over the number of summands $\sigma$ which produce any fixed monomial from $m_{\eta}$ (it is also the ratio between the sizes of the stabilizers of a monomial in $m_{\eta}$ in $S_{N}$ and in $S_{N+1}$). Using these maps we can view each $\mathbb{Q}[x_{1},\ldots,x_{N}]_{d}^{S_{N}}$ as contained in $\mathbb{Q}[x_{1},\ldots,x_{N+1}]_{d}^{S_{N+1}}$, or, more importantly, in $\Lambda_{\mathbb{Q}}$.

We set $\hat{u}_{r}:=\hat{u}_{r}^{(N)}:=\frac{N!}{(N-r)!}u_{r}$ (in $N$ variables), as well as $\hat{D}_{r}:=\frac{(N-r)!}{N!}D_{r}$, which equals the average $\sum_{I\subseteq\mathbb{N}_{N},\ |I|=r}D_{I}\big/\binom{N}{r}$ of the operators $D_{I}$ from Lemma \ref{DIexI} over all the relevant sets $I\subseteq\mathbb{N}_{N}$, and prove the following result.
\begin{prop}
The coordinates $\{\hat{u}_{r}^{(N)}\}_{r=1}^{N}$ generate the ring of symmetric functions in $N$ variables over $\mathbb{Q}$, and if we express a continuously differentiable symmetric function $\phi(x_{1},\ldots,x_{N})$ as $\hat{\varphi}(\hat{u}_{1},\ldots,\hat{u}_{N})$, then the derivative $\hat{\varphi}_{\hat{u}_{d}}$ equals $\hat{D}_{d}\phi$ for every $1 \leq d \leq N$. Moreover, for any $r\in\mathbb{N}$, the images of the symmetric function $\hat{u}_{r}^{(N)}$ in $\Lambda_{\mathbb{Q}}$ converges with $N$ to the symmetric function $(-1)^{r-1}\frac{p_{r}}{r}$, where $p_{r}$ is the usual power sum symmetric function. \label{Ntoinfty}
\end{prop}

\begin{proof}
The fact that each $\hat{u}_{r}$ is a scalar multiple of the corresponding $u_{r}$, and that $\hat{D}_{r}$ is $D_{r}$ divided by the same scalar, makes the first two statements immediate consequences of Theorem \ref{coorddef} (via Theorem \ref{egen}) and Corollary \ref{difcoord}.

Now, the dependence of $u_{r}$ on $N$ is via the definition of the normalized expressions $\tilde{e}_{h}$ and $\tilde{e}_{\lambda}$. Since $\tilde{e}_{h}$ was obtained from $e_{h}$ by dividing by the monic polynomial $\frac{N!}{(N-h)!}=\prod_{j=0}^{h-1}(N-j)$ of degree $h$ in $N$ means, it follows that when $\lambda \vdash r$, the expression $\tilde{e}_{\lambda}$ is $e_{\lambda}$ divided a the monic polynomial of degree $r$ in $N$, namely $P_{\lambda}(N):=\prod_{h=1}^{N}\prod_{j=0}^{h-1}(N-j)^{m_{h}}=\prod_{j=0}^{h-1}(N-j)^{\sum_{h=j+1}^{N}m_{h}}$. Therefore in $\hat{u}_{r}$ each expression $e_{\lambda}$ with $\lambda \vdash r$ is multiplied by the coefficient $(-1)^{l(\lambda)-1}\big(l(\lambda)-1\big)!\big/\prod_{h=1}^{r}m_{h}!$ from Theorem \ref{coorddef}, as well as by the quotient $\frac{N!/(N-r)!}{P_{\lambda}(N)}$ of two monic polynomials of degree $r$ in $N$. Since the limit of every such quotient as $N\to\infty$ exists and equals to 1, the convergence of $\hat{u}_{r}^{(N)}$ in $\Lambda_{\mathbb{Q}}$ follows.

It is also clear that the limit of the $\hat{u}_{r}^{(N)}$'s is given by the same formula from Theorem \ref{coorddef}, but with each $\tilde{e}_{\lambda}$ replaced by the ordinary $e_{\lambda}$. Moreover, Remark \ref{Bell} allows us to write this symmetric function as $-\sum_{t=1}^{r}\hat{B}_{r,t}(-e_{1},\ldots,-e_{r-t+1})/t$. But the Newton identities express $p_{r}$ as $(-1)^{r}r\sum_{t=1}^{r}\hat{B}_{r,t}(-e_{1},\ldots,-e_{r-t+1})/t$, so that this limit function is indeed the asserted one. This proves the proposition.
\end{proof}
We remark that the elementary symmetric function $e_{r}$ itself, associated with the partition of length 1 of $r$, is multiplied in $\hat{u}_{r}^{(N)}$ by 1 for every $N$ (as both the multiplier from Theorem \ref{coorddef} and the quotient of the two monic polynomials reduce to 1 for this partition). It is important to note that while the $\hat{u}_{r}^{(N)}$ tend with $N$ to $(-1)^{r-1}\frac{p_{r}}{r}$, and the equality with that limit holds for any $N$ if $r=1$, such an equality occurs for no $N$ when $r\geq2$. For example, $\hat{u}_{2}^{(N)}$ equals $e_{2}-\frac{(N-1)}{2N}e_{1}^{2}$, which for $N=2$ and the variables $x$ and $y$ becomes $\frac{xy}{2}-\frac{x^{2}+y^{2}}{4}=(x-y)^{2}/4$ (giving a multiple of the coordinate form Lemma 2.3 of \cite{[EKZ]}), and when $N=3$ (and the third variable $z$) it reduces to $(xy+xz+yz-x^{2}-y^{2}-z^{2})/3$, neither of which are multiples of $p_{2}$.

\smallskip

Recall from Remark \ref{charorddiag} that the properties from Theorem \ref{annbyders} and Corollary \ref{coordiag} are not sufficient for characterizing the $u_{r}$'s up to scalar multiples. For doing so, we consider derivatives like from those results, but now of order $r$. Since applying a derivative of order $r$ to a symmetric function of degree $r$ yields a constant, we may consider the behavior of this constant as a function of $N$, when the symmetric function is given in terms of expressions that depend on $N$, like the $\tilde{e}_{\lambda}$'s from our formulae. The proof of Proposition \ref{Ntoinfty} shows that the $\Lambda_{\mathbb{Q}}$-images of the $u_{r}$'s themselves, written as $u_{r}^{(N)}$ for emphasizing the dependence on $N$, decay like $O\big(\frac{1}{N^{r}}\big)$ as $N\to\infty$ there. This is why the normalization $\hat{u}_{r}^{(N)}$ can be more appropriate in some situations (see also Corollary \ref{totdiag} and Theorem \ref{coorgenpt} below).

Using this property we obtain the following characterization of our coordinates $\{u_{r}\}_{r=1}^{N}$, up to scalars.
\begin{thm}
For every $r\geq1$, the expression $u_{r}$ is the only linear combination of $\{\tilde{e}_{\lambda}\}_{\lambda \vdash r}$ with coefficients that do not depend on $N$ such that its image under every differential operator $\prod_{q=1}^{r}\partial_{i_{q}}$ of order $r$ in which not all the indices $i_{q}$, $1 \leq q \leq r$ decays, in the limit $N\to\infty$, as $O\big(\frac{1}{N^{r+1}}\big)$, normalized such that $\partial_{i}^{r}u_{r}=(-1)^{r-1}\frac{(r-1)!}{N^{r}}$ for some, hence any, $i\in\mathbb{N}_{N}$. \label{ordwithN}
\end{thm}
Explicit calculations suggest that for some derivatives the decay is even faster than the one established in Theorem \ref{ordwithN}---see Conjecture \ref{decay} below.

\begin{proof}
As in the proof of Theorem \ref{coorddef}, we begin with proving uniqueness, and then show that the $u_{r}$'s satisfy the required property.

Now, for evaluating the image of a combination of the $\tilde{e}_{\lambda}$'s under such a differential operator, we can again apply Lemma \ref{derdelambda}. Moreover, the same argument from the proof of Theorem \ref{annbyders} shows that for a fixed $\lambda$, we must take the $\kappa_{\nu}$'s such that $\lambda-\sum_{\nu=1}^{\gamma}\kappa_{\nu}\varepsilon_{h_{\nu}}$ gives a partition of $r-d$, where here $d=r$ so that we only have the empty partition. Thus $\lambda-\sum_{\nu=1}^{\gamma}h_{\nu}\varepsilon_{h_{\nu}}$ is also the empty partition, so that $\gamma=l(\lambda)$ and the $h_{\nu}$'s are those showing up in $\lambda$, and we also get $\kappa_{\nu}=h_{\nu}$ for every $1\leq\nu\leq\gamma=l(\lambda)$, so that the functions $\tilde{e}_{h_{\nu}-\kappa_{\nu}}$ are trivial. Therefore, every element of $X_{\iota}$ in which there are $l(\lambda)$ sets, with the set $J_{\nu}$ of indices having distinct $\iota$-images having size $h_{\nu}$ as shows up in $\lambda$ for every $1\leq\nu \leq l(\lambda)$ (in the decreasing ordering of the entries of $\lambda$, say), contributes to the expression from Lemma \ref{derdelambda} a single instance of $\prod_{\nu=1}^{l(\lambda)}\frac{m_{h_{\nu}}!}{N!/(N-h_{\nu})!}$, more simply expressed as $\prod_{h=1}^{r}\frac{m_{h}!}{N!/(N-h)!}=\big(\prod_{h=1}^{r}m_{h}!\big)\big/P_{\lambda}(N)$ in the notation from the proof of Proposition \ref{Ntoinfty}.

We therefore denote by $X_{\iota}^{\lambda}$ the subset of $X_{\iota}$ consisting of partitions of $\mathbb{N}_{d}$ into sets $\{J_{\nu}\}_{\nu=1}^{l(\lambda)}$ satisfying these conditions, and it is clear that if the indices $\{i_{q}\}_{q=1}^{r}$ have the coincidence relation such that there are $t$ distinct indices, say $j_{p}$, $1 \leq p \leq l$, and the index $j_{p}$ appears $s_{p}$ times among the $i_{q}$'s, then we express this coincidence pattern via that the partition $\sigma$ of $r$ as $\sum_{p=1}^{l}s_{p}$. It is clear that the size of $X_{\iota}^{\lambda}$ depends on $\iota$ only through the partition $\sigma$ (indeed, since our derivative of order $r$ takes any symmetric function of degree $r$ to a constant, we can evaluate this constant at any point, and applying such our derivative to any function $\phi$ gives, at such a point, the expression denoted by $\partial_{H}^{\sigma}\phi$ in the next section, where $H\subseteq\mathbb{N}_{N}$ is any subset containing all the $i_{q}$'s). Now, we have $|X_{\iota}^{\lambda}|>0$ if and only if the indices with pattern $\sigma$ can be organized in sets of disjoint indices whose sizes are determined by $\lambda$, which is easily seen to be equivalent to $\sigma^{t}$ dominating $\lambda$ in the natural dominance order on partitions of $r$, a statement that we shall denote by $\lambda\leq\sigma^{t}$.

Consider thus a linear combination $\sum_{\lambda \vdash r}c_{\lambda}\tilde{e}_{\lambda}$, with coefficients $\{c_{\lambda}\}_{\lambda \vdash r}$ that are independent of $N$, not all of which are 0. Take a differential operator $\prod_{q=1}^{r}\partial_{i_{q}}$, corresponding to $\iota:\mathbb{N}_{r}\to\mathbb{N}_{N}$ and thus to a partition $\sigma \vdash r$, and then its action sends our combination to $\sum_{\lambda \vdash r}c_{\lambda}|X_{\iota}^{\lambda}|\big(\prod_{h=1}^{r}m_{h}!\big)\big/P_{\lambda}(N)$, where we saw that we can restrict the sum to be taken only over $\lambda$ that are dominated by $\sigma^{t}$. Moreover, as all the numerators are independent of $N$, and all the denominators are monic polynomials of degree $r$ in $N$, this expression is $\sum_{\lambda\leq\sigma^{t}}c_{\lambda}|X_{\iota}^{\lambda}|\big(\prod_{h=1}^{r}m_{h}!\big)\big/N^{r}+O\big(\frac{1}{N^{r+1}}\big)$. We are therefore interested in those combinations $\sum_{\lambda \vdash r}c_{\lambda}\tilde{e}_{\lambda}$ for which the last numerator vanishes for every $\sigma$ which is not the maximal partition in terms of dominance, of length 1 (i.e., where $\sigma^{t}$ is not the minimal partition, of length $r$).

Consider first the case where $\sigma$ is such that $\sigma^{t}$ is one of the (possibly several) minimal partitions $\lambda \vdash r$ for which $c_{\lambda}\neq0$. Then our numerator is the sum over one non-zero element, meaning that the order of decay of our derivative with $N$ is precisely as $O\big(\frac{1}{N^{r}}\big)$. Under our assumption this is allowed only if $\sigma$ is the maximal partition, corresponding to the derivative of order $r$ with respect to a single variable, and therefore the multiplier of the element $\tilde{e}_{1}^{r}$ corresponding to the minimal partition cannot vanish. For this $\sigma$ the corresponding derivative is $\partial_{i}^{r}u_{r}$, sending any combination to $\frac{r!}{N^{r}}$ times the coefficient of $\tilde{e}_{1}^{r}$, so that in our normalization this coefficient has to be $(-1)^{r-1}/r$, and we have to see that our condition determines the remaining coefficients. We do this by induction using the dominance order, take a partition $\lambda \vdash r$ that is not the one that we just normalized, and assume that we have already determined all the coefficients $c_{\eta}$ for $\eta\leq\lambda$ except for $c_{\lambda}$ itself. As the vanishing of the numerator arising from the derivative associated with $\sigma=\lambda^{t}$ expresses a positive multiple of $c_{\lambda}$ in terms of the coefficients $c_{\eta}$ that we already know, the uniqueness statement follows (in fact, so does the existence, but without the explicit form of the combination satisfying it).

It thus remains to prove that our combination $u_{r}$ does have this property. After multiplying by $\frac{N!}{(N-r)!}$, this is equivalent, for rational functions of $N$, to all the derivatives of order $r$ of the expression $\hat{u}_{r}^{(N)}$ from Proposition \ref{Ntoinfty} vanishing as $N\to\infty$, and the derivative with respect to a single variable tending to $(-1)^{r-1}(r-1)!$. But that proposition shows that in this limit, $\hat{u}_{r}^{(N)}$ tends to $(-1)^{r}\frac{p_{r}}{r}$, so that its derivatives tend to those of the latter symmetric function. Recalling that $p_{r}$ is the power sum function $\sum_{i=1}^{N}x_{i}^{r}$ (i.e., it is a trace function in the terminology of \cite{[B]}), it is indeed annihilated by any mixed derivative, and the pure derivative indeed gives the required (non-vanishing) value. This completes the proof of the theorem.
\end{proof}

One can determine the expression for all the derivatives of order $r$ of $u_{r}$ for small $r$ explicitly. Based on doing so for all $r\leq6$ and $\iota:\mathbb{N}_{r}\to\mathbb{N}_{N}$ (or more precisely $\sigma \vdash r$), we pose the following conjecture
\begin{conj}
Let $\prod_{q=1}^{r}\partial_{i_{q}}$ be a differential operator of order $r$ that corresponds to a partition $\sigma \vdash r$ as above. Then the total expression of $\partial^{d}u_{r}\big/\prod_{q=1}^{d}\partial_{i_{q}}$ decays like $O\big(\frac{1}{N^{r+l(\sigma)-1}}\big)$ as $N\to\infty$. \label{decay}
\end{conj}
Note that while for Theorem \ref{ordwithN} it suffices to consider the limit of $\hat{u}_{r}^{(N)}$ as $N\to\infty$, for doing the finer analysis required for investigating Conjecture \ref{decay} in a similar way one needs to remove the parts associated with partitions that are dominating $\sigma$. As, for example, subtracting $(-1)^{r}\frac{p_{r}}{r}\big/N^{r}$ and subtracting $(-1)^{r}\frac{p_{r}}{r}\big/\frac{N!}{(N-r)!}$ from $u_{r}$ and multiplying by $N^{r+1}$ (or by $\frac{N!}{(N-r-1)!}$) would give different limits as $N\to\infty$, this seems to be a more delicate question, that is therefore left for future research.

\section{Derivatives at Diagonal Points \label{Diag}}

The expression $D_{d}\phi$ from Corollary \ref{difcoord} is only defined at points where the variables $x_{i}$, $i\in\mathbb{N}_{N}$ have distinct values. We now turn to evaluate these derivatives in case some (or all) of the values of the $x_{i}$'s coincide. We shall assume throughout that $\phi$ is a symmetric function of $N$ real variables $x_{i}$, $i\in\mathbb{N}_{N}$ (expressed as $\varphi(u_{1},\ldots,u_{N})$), and that $\phi$ (and $\varphi$) have the derivatives of every order that will show up. All the results will hold equally well for holomorphic symmetric functions of $N$ complex variables, but since smooth functions of complex variables can be non-holomorphic, the real case is more natural for presenting our results.

\smallskip

We begin by observing the symmetry in the derivatives of $\phi$ where points coincide.
\begin{lem}
Assume that $J\subseteq\mathbb{N}_{N}$ is a set of size $p$, and that we differentiate at a point where the values of all the variables $x_{i}$ with $i \in J$ is the same. Assign, to each $i \in J$, a multiplicity $h_{i}$, and set $g:=\sum_{i \in J}h_{i}$. Then the value of the derivative $\partial^{g}\phi\big/\prod_{i \in J}\partial x_{i}^{h_{i}}$ is invariant under the action of the symmetric group of $J$ on the sets of multiplicities. \label{dersym}
\end{lem}

\begin{proof}
Consider first the case where $J$ is the set $\{i,j\}$, of size 2. For every $\delta\neq0$, the value of $\phi$ obtained by adding $\delta$ to $x_{i}$ and leaving the other variables invariant coincides with the value attained when $x_{j}$ is taken to $x_{j}+\delta$ and the rest of the variables are left invariant (by symmetry). Therefore the limits defining the derivatives $\phi_{i}$ and $\phi_{j}$ at our point are the same, as desired. Similarly, the limit defining $\phi_{i}$ at the point where $x_{i}$ is replaced by $x_{i}+\delta$ coincides (again by symmetry) with the one producing $\phi_{j}$ with $x_{j}+\delta$ in place of $x_{j}$. As this is the case for all $\delta$, we obtain $\phi_{ii}=\phi_{jj}$. More generally, the value of $\partial_{i}^{h-1}\phi$ with $x_{i}+\delta$ instead of $x_{i}$ will be the same as that of $\partial_{j}^{h-1}\phi$ when $x_{j}$ is replaced by  $x_{j}+\delta$, so that by the taking the limit we established the result when $|J|=2$ and one of the multiplicities vanishes.

Let now $J$ and the multiplicities by general. Since the permutation group is generated by transpositions, take two indices $i$ and $j$ from $J$, and it suffices to prove that $\partial^{g}\phi\big/\prod_{i \in J}\partial x_{i}^{h_{i}}$ is invariant under the transposition interchanging $i$ and $j$. Assume, without loss of generality, that $h_{i} \geq h_{j}$, let $h:=h_{i}-h_{j}$, and consider the function $\psi:=\partial^{g-h}\phi\big/\partial x_{i}^{h_{j}}\partial x_{j}^{h_{j}}\prod_{k \in J\setminus\{i,j\}}\partial x_{k}^{b_{k}}$. Since $\phi$ and the differential operator taking it to $\psi$ are symmetric under interchanging $i$ and $j$, so is $\psi$, and thus the previous paragraph gives us $\partial_{i}^{h}\psi=\partial_{j}^{h}\psi$. But as the latter two functions are the asserted one $\partial^{g}\phi\big/\prod_{i \in J}\partial x_{i}^{h_{i}}$ and its image under our transposition respectively, the desired equality follows. This proves the lemma.
\end{proof}
Note that for second derivatives, Lemma \ref{dersym} gives $\phi_{ii}=\phi_{jj}$ when $x_{i}=x_{j}$ (as well as $\phi_{ik}=\phi_{jk}$ for a third, distinct index $k$), but these derivatives are not related to $\phi_{ij}$. As one example for this, let $\phi$ be a \emph{trace function} in the terminology of \cite{[B]}, i.e., a function of the form $\sum_{i=1}^{N}f(x_{i})$ for a function $f$ of a single variable. Then $\phi_{ij}=0$ but $\phi_{ii}$ and $\phi_{jj}$ equal $f''(x_{i})$ (and $x_{i}=x_{j}$), which need not vanish. As another example, take $\phi$ be the symmetric function $e_{h}$ for some $h\geq2$. Then it is clear (e.g., using Lemma \ref{expeh}) that $\phi_{ij}=e_{h-2}(x_{I^{c}})\neq0$ for $I=\{i,j\}$, while $\phi_{ii}=\phi_{jj}=0$.

It follows from Lemma \ref{dersym} that if $\sigma$ is the partition of the total order $g$ of the derivative mentioned there into the parts $\{h_{i}\}_{i \in J}$ (with $l(\sigma)\leq|J|$), written classically as $g=\sum_{q=1}^{l}h_{q}$ (ignoring the vanishing $h_{i}$'s, say), then by assigning the multiplicities $h_{q}$, $1 \leq q \leq l$ to arbitrary distinct indices from $J$ the resulting derivative $\partial^{g}\phi\big/\prod_{i \in J}\partial x_{i}^{h_{i}}$ of $\phi$ depends only on $J$ and $\sigma$, and not on the assignment. We can thus denote this derivative by $\partial_{J}^{\sigma}\phi$, and it is well-defined at points like in that lemma (we can define it using an arbitrary assignment, or using an average, at every point in $\mathbb{R}^{N}$, but we shall consider such expressions only under the assumption that the values of the variables $x_{i}$, $i \in J$ are all the same, where we now know it to be well-defined and independent of any additional data). In particular $\partial_{J}^{(1)}\phi$ is just $\phi_{i}$ for $i \in J$, where the superscript stands for the unique partition of 1, and this is clearly the only simple derivative of order 1 that we can define based on $J$. Similarly, if $J$ contains at at least two elements and $i$ and $j$ are in $J$ then the derivatives $\phi_{ij}$ and $\phi_{ii}$, whose difference was mentioned in the Introduction, can be written as $\partial_{J}^{(1,1)}\phi$ and $\partial_{J}^{(2)}\phi$ respectively. The combination of third derivatives that shows up there, when $|J|\geq3$, is therefore $\partial_{J}^{(1,1,1)}\phi-\frac{3}{2}\partial_{J}^{(2,1)}\phi+\frac{1}{2}\partial_{J}^{(3)}\phi$ in this notation.

\smallskip

The combinations of higher-order derivatives were obtained from the action of an operator like $D_{I}$ at points where some variables coincide, and due to the vanishing denominators, they can be obtained by taking the limit of the expressions for points with distinct variables as some variables tend to the value of others. We will see that when the values of $p$ variables coincide, we shall obtain, in general, derivatives of all orders $1 \leq g \leq p$, and the form of the $g$th derivative is always the same. This means that for every $g\in\mathbb{N}$ (recall that in this paper the notation $\mathbb{N}$ only includes positive integers, without 0, like the set $\mathbb{N}_{N}$ from above) there are fixed coefficients $\{c_{\sigma}\}_{\sigma \vdash g}$ such that wherever $J$ is a set of $p$ indices, with $p \geq g$, the derivatives of order $g$ of $\phi$ that are based on indices from $J$ is always the same combination $\partial_{J}^{g}\phi=\sum_{\sigma \vdash g}c_{\sigma}\partial_{J}^{\sigma}\phi$ of the derivatives $\partial_{J}^{\sigma}\phi$ with $\sigma \vdash g$. The derivative $\partial_{J}^{1}\phi$ is just $\partial_{J}^{(1,1)}\phi=\phi_{i}$ from above, $\partial_{J}^{2}\phi$ will be the difference $\partial_{J}^{(1,1)}\phi-\partial_{J}^{(2)}\phi$ from above, and the expression from the end of the previous paragraph will be $\partial_{J}^{3}\phi$.

For determining the explicit value of $\partial_{J}^{g}\phi$, we shall need some notation. Let a finite set $K$ of indices be given, as well as a non-negative integer $\nu$. We then define the set $A_{K,\nu}$ to be the set of tuples $(a,\vec{b})$, with $a\in\mathbb{N}$ and $\vec{b}=\{b_{k}\}_{k \in K}\in\mathbb{N}^{K}$, satisfying the equality $a+\sum_{k \in K}b_{k}=|K|+\nu+1$. Note that $K$ can be empty, and then $A_{\nu}$ only consists of the single 1-tuple with entry $a=\nu+1$. Also note that for $\nu=0$ the set $A_{K,\nu}$ contains again only one element, in which $a$ and all the $b_{k}$'s equal 1.

We now aim to prove the following result.
\begin{thm}
Assume that $I$ is a set containing $J$, with $|I|=d$ and $|J|=p\geq1$, and consider a point where the variables $x_{i}$, $i \in J$ all take the value $y$, and that the values of the variables $x_{i}$, $i \in I \setminus J$ are distinct from $y$ and from one another. At such a point set \[\partial_{J}^{g}\phi:=\sum_{\sigma \vdash g}(-1)^{g-l(\sigma)}\frac{g\cdot\big(l(\sigma)-1\big)!}{\prod_{h=1}^{g}h!^{m_{h}}m_{h}!}\partial_{J}^{\sigma}\phi\] for every $1 \leq g \leq p$ (independently of $I$), where for $\sigma \vdash g$ and $1 \leq h \leq p$ the symbol $m_{h}$ stands for the multiplicity with which $h$ appears in $\sigma$, and we have the equality \[D_{I}\phi=\sum_{k \in I \setminus J}\frac{\phi_{k}}{(y-x_{k})^{p}\prod_{l \in I\setminus(J\cup\{k\})}(x_{l}-x_{k})}+\sum_{(a,\vec{b}) \in A_{I \setminus J,p-1}}\frac{(-1)^{p-a}\partial_{J}^{a}\phi}{\prod_{k \in I \setminus J}(x_{k}-y)^{b_{k}}}.\] \label{expordp}
\end{thm}

\begin{rmk}
The derivative $\partial_{J}^{\sigma}\phi$ is the value at our point of the image of $\phi$ under an operator $\prod_{q=1}^{g}\partial_{i_{q}}$ for some choice of $i_{q}$'s such that the associated map $\iota:\mathbb{N}_{g}\to\mathbb{N}_{n}$ as in Lemma \ref{derdelambda} has image in $J$ and corresponds to $\sigma$. Choosing $g$ distinct elements of $J$, say $j_{h}$, $1 \leq h \leq g$, and normalizing the maps $\iota$ yielding the different $\sigma \vdash r$ according to these $j_{h}$'s, this means that in the terminology from Remark \ref{Bell}, we can express the operator $\partial_{J}^{g}$ in terms of the usual (exponential) partial Bell polynomials $B_{g,t}$ in the derivatives $\partial_{j_{h}}$, $1 \leq h \leq g$ as $(-1)^{g}g\sum_{t=1}^{g}B_{g,t}(-\partial_{j_{1}},\ldots,-\partial_{j_{g-t+1}})$. \label{Bellexp}
\end{rmk}

We will prove Theorem \ref{expordp} by induction on $p$. For this we shall need the following preliminary calculations.
\begin{lem}
Let $\psi$ be a function of the variable $x_{j}$ that is differentiable at least $\nu$ times, and let $\{x_{k}\}_{k \in K}$ be distinct real numbers, where $K$ is a finite set of indices not containing $j$. Recall the set $A_{K,\nu}$ from above, and then the equality \[\frac{d^{\nu}}{dx_{j}^{\nu}}\frac{\psi'(x_{j})}{\prod_{k \in K}(x_{k}-x_{j})}=\sum_{(a,\vec{b}) \in A_{K,\nu}}\frac{\nu!\psi^{(a)}(x_{j})}{(a-1)!\prod_{k \in K}(x_{k}-x_{j})^{b_{k}}}\] holds wherever $x_{j}$ is different from the values of the $x_{k}$'s. \label{derwithdenom}
\end{lem}
Lemma \ref{derwithdenom} is clear when $K$ is empty, as it reduces to the immediate equality $\frac{d^{\nu}}{dx_{j}^{\nu}}\psi'(x_{j})=\psi^{(\nu+1)}(x_{j})$. However, we shall require it for non-empty $K$ later.

\begin{proof}
We argue by induction on $\nu$, with the case $\nu=0$ being trivial, since the set $A_{K,0}$ only contains the tuple with $a=1$ and $b_{k}=1$ for every $k \in K$ (note the initial derivative of $\psi$ in the numerator). Now, assume that the result holds for some $\nu$, and since differentiating the numerator in one summand for the $\nu$th derivative increases $a$ by 1 and differentiating $\frac{1}{(x_{k}-x_{j})^{b_{k}}}$ increases $b_{k}$ by 1, we indeed obtain a sum over the expressions associated with $A_{K,\nu+1}$, and we only have to check the multiplying coefficients. So take an element $(a,\vec{b}) \in A_{K,\nu+1}$, and we saw that the derivatives producing multiples of this term are obtained from terms associated with tuples in $A_{K,\nu+1}$ in which exactly one of the entries is one less than the corresponding entry of $(a,\vec{b})$, and the other entries are like those of $(a,\vec{b})$. Now, from the term with $b_{k}-1$ the derivative of $\frac{1}{(x_{k}-x_{j})^{b_{k}-1}}$ gives our term multiplied by $b_{k}-1$ (thus covering the case with $b_{k}=1$ not appearing in this construction), and all of these summands showed up in the $\nu$th derivative with the coefficient $\frac{\nu!}{(a-1)!}$. The last contribution is, when $a>1$, from the term with $a-1$, where there is no coefficient from the derivative but the multiplier is $\frac{\nu!}{(a-2)!}=\frac{\nu!}{(a-1)!}(a-1)$ (which indeed vanishes when $a=1$ and we can ignore this restriction as well). The total numerical coefficient is thus $\frac{\nu!}{(a-1)!}$ times $a-1+\sum_{k \in K}(b_{k}-1)$, and as the latter sum equals $\nu+|K|+2-1-|K|=\nu+1$ for $(a,\vec{b}) \in A_{K,\nu+1}$, we indeed obtain the desired coefficient $\frac{(\nu+1)!}{(a-1)!}$. Thus proves the lemma.
\end{proof}

\smallskip

Recall that the definition of $\partial_{J}^{p}\phi$ from Theorem \ref{expordp} remains the same when $J$ is increased to a larger set. In particular, the expressions $\partial_{J}^{h}\phi$ are defined for every $1 \leq h \leq p$. The next relation that we shall need is among these expressions.
\begin{lem}
Assume that all the variables $x_{i}$ with $i \in J$ equal $y$, and choose some $j$ which is not in $J$. Take some $1 \leq g \leq p$, and then as $x_{j} \to y$ as well, the limit of the expression $\partial_{j}^{g}\phi+\sum_{\mu=0}^{g-1}\frac{(g-1)!}{\mu!}(-1)^{g-\mu}\partial_{j}^{\mu}\partial_{J}^{g-\mu}\phi$ vanishes. Equivalently, the expression $\partial_{j}^{g}\phi+\sum_{\mu=1}^{g-1}\frac{(g-1)!}{\mu!}(-1)^{g-\mu}\partial_{j}^{\mu}\partial_{J}^{g-\mu}\phi$ takes the limit $-(-1)^{g}(g-1)!\partial_{J}^{g}\phi$ as $x_{j} \to y$ for every such $g$, and when we consider this expression as a multiple of $\partial_{J\cup\{j\}}^{g}\phi$, the latter statement holds also for $g=p+1$. \label{relsders}
\end{lem}

\begin{proof}
For each $\mu\geq1$ we write $\partial_{J}^{g-\mu}\phi$, via the definition in Theorem \ref{expordp}, as the sum $\sum_{\tau \vdash g-\mu}(-1)^{g-\mu-l(\tau)}(g-\mu)\cdot\big(l(\sigma)-1\big)!\partial_{J}^{\tau}\phi\big/\prod_{h=1}^{g-\mu}h!^{\tilde{m}_{h}}\tilde{m}_{h}!$, where for $\tau$ we use $\tilde{m}_{h}$ for the multiplicities. When we differentiate it $\mu$ times with respect to $j$, at the limit $x_{j} \to y$ the derivative $\partial_{j}^{\mu}\partial_{J}^{\tau}$ becomes, by definition, $\partial_{J\cup\{j\}}^{\sigma}$ where $\sigma$ is the partition of $g$ such that $\sigma-\mu\varepsilon_{\mu}=\tau$ in the notation used in, e.g., Corollary \ref{difepart}. We thus take a partition $\sigma \vdash g$, and see with which multiplicity the expression $\partial_{J\cup\{j\}}^{\sigma}\phi$ shows up in the limit of $\partial_{j}^{g}\phi+\sum_{\mu=1}^{g-1}\frac{(g-1)!}{\mu!}(-1)^{g-\mu}\partial_{j}^{\mu}\partial_{J}^{g-\mu}\phi$ as $x_{j} \to y$. The first term only produces the derivative corresponding to the partition of $g$ as $g$ (with length 1), and as for every $1\leq\mu \geq g-1$ both $j$ and indices different from $j$ show up in $\partial_{j}^{\mu}\partial_{J}^{\tau}\phi$ for any $\tau \vdash g-\mu$, the sum over $\mu$ only contributes to partitions $\sigma \vdash g$ with $l(\sigma)\geq2$.

Now, given such $\sigma$ we have $l(\sigma^{t}) \leq g-1$, and we get a contribution exactly from those $1\leq\mu \leq g-1$ and $\tau \vdash g-\mu$ such that $\tau=\sigma-\mu\varepsilon_{\mu}$, namely from those indices $\mu$ that show up in $\sigma$, and then $\tau$ is determined. When $m_{h}$ (resp. $\tilde{m}_{h}$) is the multiplicity with which $h$ shows up in $\sigma$ (resp. $\tau$), we get that $\tilde{m}_{h}$ can be expressed succinctly as $m_{\mu}-\delta_{h,\mu}$, and $l(\tau)=l(\sigma)-1$. Moreover, we can write $1\big/\prod_{h=1}^{g-\mu}h!^{\tilde{m}_{h}}\tilde{m}_{h}!$, expressing the denominator of the numerical coefficient multiplying $\partial_{J}^{\tau}\phi$, as $\mu!m_{\mu}\big/\prod_{h=1}^{g-1}h!^{m_{h}}m_{h}!$ using the multiplicities for $\sigma$ (and we can take the product over $h$ to go up to $g$, since $m_{g}=0$ when $\sigma \vdash g$ and $l(\sigma)\geq2$). Altogether, since the part $(-1)^{g-\mu}$ of the sign with which $\partial_{J}^{\tau}\phi$ appears in $\partial_{J}^{g-\mu}\phi$ cancels with the sign multiplier of the $\mu$th summand, and the $\mu!$ in the latter expression cancels with the denominator of $\frac{(g-1)!}{\mu!}$, the total coefficient with which $\partial_{J\cup\{j\}}^{\sigma}\phi$ shows up in our expression is $(g-1)!$ times \[-\frac{(-1)^{l(\sigma)}\big(l(\sigma)-2\big)!}{\prod_{h=1}^{g-\mu}h!^{m_{h}}m_{h}!}\bigg[\sum_{\mu=1}^{g-1}(g-\mu)m_{\mu}\bigg]=-(-1)^{g}\frac{(-1)^{g-l(\sigma)}g\big(l(\sigma)-1\big)!}{\prod_{h=1}^{g-\mu}h!^{m_{h}}m_{h}!}\] (this is so, since the vanishing of the multiplier $m_{\mu}$ when $\mu$ does not show up in $\sigma$ allows us to consider the sum over all $1\leq\mu \leq g-1$, and by adding $\mu=g$ in the same trivial manner we recall that $\sum_{\mu=1}^{g}m_{\mu}=l(\sigma)$ and $\sum_{\mu=1}^{g}m_{\mu}\mu=g$ when $\sigma \vdash g$). Note that in the remaining partition of length 1 the latter expression (including the multiplier $(g-1)!$) reduces to the coefficient 1 from the previous paragraph. As this is $-(-1)^{g}(g-1)!$ times the coefficient multiplying $\partial_{J\cup\{j\}}^{\sigma}\phi$ in $\partial_{J\cup\{j\}}^{g}\phi$ for every $\sigma \vdash g$, now without the restriction on $l(\sigma)$ (also when $g=p+1$), we obtain the second assertion, from which the first one follows since when $g \leq p$ we can represent the summand $(g-1)!(-1)^{g}\partial_{J}^{g}\phi$ cancels with $-(g-1)!(-1)^{g}\partial_{J\cup\{j\}}^{g}\phi$ by Lemma \ref{dersym}. This proves the lemma.
\end{proof}

The combination of these results that we shall need is the following one.
\begin{prop}
At a point like in Lemma \ref{relsders}, with $\phi$ differentiable at least $p$ times, take an index set $K$ that is disjoint from $J\cup\{j\}$, and let $\{x_{k}\}_{k \in K}$ be as in Lemma \ref{derwithdenom}. Then the function \[\frac{\phi_{j}}{\prod_{k \in K}(x_{k}-x_{j})}+\sum_{c=0}^{p-1}(x_{j}-y)^{c}\sum_{(a,\vec{b}) \in A_{K,c}}\frac{(-1)^{a}\partial_{J}^{a}\phi}{\prod_{k \in K}(x_{k}-y)^{b_{k}}}\] of $x_{j}$, and all of its derivatives up to order $p-1$, vanish at the limit $x_{j} \to y$. The value of its $p$th derivative is $-p!\sum_{(a,\vec{b}) \in A_{K,p}}(-1)^{a}\partial_{J}^{a}\phi\big/\prod_{k \in K}(x_{k}-y)^{b_{k}}$. \label{LHopital}
\end{prop}

\begin{proof}
Take some $\nu\geq0$, and consider the $\nu$th derivative of this function of $x_{j}$. The derivative of the first term is given in Lemma \ref{derwithdenom}, where $\psi$ is the restriction of $\phi$ to the variable $x_{j}$ with the other ones fixed with our conditions, so that each $\psi^{(a)}(x_{j})$ is $\partial_{j}^{a}\phi$. In the second term we only differentiate the powers of $x_{j}-y$ and the numerators $\partial_{J}^{a}\phi$, since the denominators are independent of $x_{j}$. The Generalized Leibniz' Rule shows that the $\nu$th derivative of $(x_{j}-y)^{c}\partial_{J}^{a}\phi$ is $\sum_{\mu}\binom{\nu}{\mu}\frac{c!}{(c-\nu+\mu)!}(x_{j}-y)^{c-\nu+\mu}\partial_{j}^{\mu}\partial_{J}^{a}\phi$, with $\max\{0,\nu-c\}\leq\mu\leq\nu$. Interchanging the sum over $c$ and $\mu$, and replacing $c$ by $c+\nu-\mu$, we find that the $\nu$th derivative of the second term is \[\sum_{\mu=\nu+1-p}^{\nu}\binom{\nu}{\mu}\sum_{c=0}^{p-1-\nu+\mu}(x_{j}-y)^{c}\frac{(c+\nu-\mu)!}{c!}\sum_{(a,\vec{b}) \in A_{K,c+\nu-\mu}}\frac{(-1)^{a}\partial_{j}^{\mu}\partial_{J}^{a}\phi}{\prod_{k \in K}(x_{k}-y)^{b_{k}}},\] where the binomial coefficient restricts $\mu$ to $\max\{0,\nu+1-p\}\leq\mu\leq\nu$. As the terms with $c>0$ clearly vanish as $x_{j} \to y$, it remains to consider the limit of \[\sum_{(a,\vec{b}) \in A_{K,\nu}}\frac{\nu!\partial_{j}^{a}\phi}{(a-1)!\prod_{k \in K}(x_{k}-x_{j})^{b_{k}}}+\sum_{\mu=\nu+1-p}^{\nu}\frac{\nu!}{\mu!}\sum_{(a,\vec{b}) \in A_{K,\nu-\mu}}\frac{(-1)^{a}\partial_{j}^{\mu}\partial_{J}^{a}\phi}{\prod_{k \in K}(x_{k}-y)^{b_{k}}}\] (again with the sum over $\mu$ effectively being over $\max\{0,\nu+1-p\}\leq\mu\leq\nu$), where we can substitute $x_{j}=y$ in the denominators of the summands in the first term.

Take some values of $\{b_{k}\}_{k \in K}$ that show up in an element of $A_{K,\nu}$. The value of $a$ in the corresponding summand in the first term is $|K|+\nu+1-\sum_{k \in K}b_{k}$, an expression that we denote by $g$, and satisfies $g\geq1$ since $(g,\vec{b}) \in A_{K,\nu}$ by assumption. The values of $\mu$ for which this element shows up in $A_{K,\nu-\mu}$ are those for which the required value of $a$, which is now $g-\mu$, is positive, namely $\mu \leq g-1$. Assuming that $\nu \leq p-1$, so that the effective lower bound on $\mu$ is indeed 0, these terms in our expression are $\nu!\big/(g-1)!\prod_{k \in K}(x_{k}-y)^{b_{k}}$ times the combination $\partial_{j}^{g}\phi+\sum_{\mu=0}^{g-1}\frac{(g-1)!}{\mu!}(-1)^{g-\mu}\partial_{j}^{\mu}\partial_{J}^{g-\mu}\phi$, and $g\leq\nu+1 \leq p$. As this combination vanishes as $x_{j} \to y$ by the first assertion in Lemma \ref{relsders}, this establishes the first assertion. If we now take $\nu=p$, then the combination that we get is the same, but with the sum over $\mu$ starting from 1, and the bound $g \leq p+1$. We apply the second assertion of Lemma \ref{relsders} for the limit as $x_{j} \to y$, and merge with the external multiplier (with $\nu=p$) to get $(-1)^{g+1}p!\partial_{J}^{g}\phi\big/\prod_{k \in K}(x_{k}-y)^{b_{k}}$. As $g$ was seen to be the value of $a$ complementing the $b_{k}$'s to an element of $A_{K,\nu}=A_{K,p}$, the second assertion follows as well. This proves the proposition.
\end{proof}

\begin{proof}[Proof of Theorem \ref{expordp}]
We prove the desired formula by induction on $p$. The case $p=1$ is just the definition of $D_{I}\phi$ where all the points are distinct, since $A_{I \setminus J,0}$ has a single simple element and $\partial_{J}^{1}\phi=\phi_{i}$ where $J$ is the singleton $\{i\}$. We shall assume that $d>p$, and, by induction, that the formula for $D_{I}\phi$ is valid when $|J|=p$. We take some index $j \in I \setminus J$, view the values of the variable $x_{k}$ for $k \in I\setminus(J\cup\{j\})$ as fixed, and allow the variable $x_{j}$ to tend to the common value $y$ of $x_{i}$, $i \in J$. In the terms involving $\phi_{k}$ for $k \in I\setminus(J\cup\{j\})$, this limit is obtained by simply replacing the multiplier $x_{j}-x_{k}$ by another power of $y-x_{k}$, thus yielding the desired sum over such $k$ in $D_{I}\phi$ for $p+1$.

For analyzing the remaining terms, set $K:=I\setminus(J\cup\{j\})$, so that $I \setminus J$ is $K\cup\{j\}$, and note that the value of $b_{j}$ in an element of $A_{I \setminus J,p-1}$ lies between 1 and $\nu$, and that when we fix that value, the remaining entries give an element of $A_{K,p-b_{j}}$. Having the denominator $(y-x_{j})^{p}$ in the term involving $\phi_{j}$, we replace $b_{j}$ by $c=p-b_{j}$, and the equality that we have to show that the limit of
\[\frac{1}{(x_{j}-y)^{p}}\Bigg[\frac{(-1)^{p}\phi_{j}}{\prod_{k \in K}(x_{k}-x_{j})}+\sum_{c=0}^{p-1}(x_{j}-y)^{c}\sum_{(a,\vec{b}) \in A_{K,c}}\frac{(-1)^{p-a}\partial_{J}^{a}\phi}{\prod_{k \in K}(x_{k}-y)^{b_{k}}}\Bigg]\] as $x_{j} \to y$ exists and equals $\sum_{(a,\vec{b}) \in A_{K,p}}(-1)^{p+1-a}\partial_{J\cup\{j\}}^{a}\phi\big/\prod_{k \in K}(x_{k}-y)^{b_{k}}$. We claim that we can do it by applying L'H\^{o}pital's Rule $p$ successive times.

Indeed, the denominator is $(x_{j}-y)^{p}$, which vanishes along with its first $p-1$ derivatives as $x_{j} \to y$, and the numerator is $(-1)^{p}$ times the function from Corollary \ref{LHopital}, which also has the same vanishing properties. Therefore the limit is the quotient of the $p$th derivative of the numerator divided by that of the denominator, which is $p!$. As $\frac{(-1)^{p}}{p!}$ times the expression for the $p$th derivative from Proposition \ref{LHopital} is the asserted value, the result required for the induction step follows. This proves the theorem.
\end{proof}
It is clear from the proof of Theorem \ref{expordp} that the only property of the explicit expression for $\partial_{J}^{p}\phi$ is the relation from Lemma \ref{relsders}. As this lemma expresses $\partial_{J}^{g}\phi$, and with it $\partial_{J\cup\{j\}}^{g}\phi$, in terms of the derivatives of lower order (and their derivatives with respect to $x_{j}$), this relation can be used for constructing these expressions inductively.

\smallskip

Theorem \ref{expordp} considers one operator $D_{I}$, with $|I|=d$, at a point where the the values of the variables associated with one set $J$, of size $p$, coincide, while all the other values are distinct. This is, indeed, the most generic assumptions under which an expression like $\partial_{J}^{p}\phi$, for this value of $p$, shows up. We can view this situation as a partition of $I$ into $d-p+1$ sets, one of which is $J=J_{1}$ and the other ones are singletons $J_{\alpha}$, $2\leq\alpha \leq d-p+1$, and denote by $y_{\alpha}$ the value of $x_{i}$ for $i \in J_{\alpha}$ (this is $y$ when $\alpha=1$ and $x_{k}$ if $\alpha>1$ and $J_{\alpha}$ is the singleton $\{k\}$). All the terms in the expression for $D_{I}\phi$ in Theorem \ref{expordp} involve (up to a multiplying numerical coefficient which is a sign) a derivative $\partial_{J_{\alpha}}^{g}\phi$ with $1 \leq g\leq|J_{\alpha}|$ (this means $g=1$ for $\alpha>1$), over a denominator which is a product over $\beta\neq\alpha$ of $(y_{\beta}-y_{\alpha})^{c_{\beta}}$, where $c_{\beta}\geq|J_{\beta}|$ and $g+\sum_{\beta\neq\alpha}c_{\beta}=d$.

At a general point in $\mathbb{R}^{N}$ we may well have several sets of points coinciding. For generalizing the interpretation of Theorem \ref{expordp} from the previous paragraph, we consider a collection $\vec{J}$ of non-empty disjoint finite sets $\{J_{\alpha}\}_{\alpha=1}^{s}$ and an index $1\leq\alpha \leq s$, and define the set $B_{\vec{J},\alpha}$ to be the set of tuples of positive integers the form $(a,\vec{c})$ in which $\vec{c}=\{c_{\beta}\}_{\beta\neq\alpha}$ is such that $c_{\beta}\geq|J_{\beta}|$ for every $\beta\neq\alpha$, and $a\in\mathbb{N}$ satisfies $a+\sum_{\beta\neq\alpha}c_{\beta}=\sum_{\alpha=1}^{s}|J_{\alpha}|$. Note that when $J_{\beta}$ is a singleton for every $\beta\neq\alpha$ and $I=\bigcup_{\alpha=1}^{s}J_{\alpha}$, the set $B_{\vec{J},\alpha}$ is the same as $A_{I \setminus J_{\alpha},|J_{\alpha}|-1}$ (as it should be, by comparing Proposition \ref{DIgenpt} below to Theorem \ref{expordp}), so that it is indeed a more general definition. Using these sets, we obtain the following formula for $D_{I}\phi$ at any point.
\begin{prop}
Assume that the index set $I$ is the disjoint union of non-empty sets $\{J_{\alpha}\}_{\alpha=1}^{s}$, and consider a point in $\mathbb{R}^{N}$ at which $x_{i}$ takes the same value $y_{\alpha}$ for every $i \in J_{\alpha}$, but the values $\{y_{\alpha}\}_{\alpha=1}^{s}$ are distinct. At such a point we have \[D_{I}\phi=\sum_{\alpha=1}^{s}\sum_{(a,\vec{c}) \in B_{\vec{J},\alpha}}\Bigg[\prod_{\beta\neq\alpha}\binom{c_{\beta}-1}{|J_{\beta}|-1}\Bigg]\frac{(-1)^{|J_{\alpha}|-a}\partial_{J_{\alpha}}^{a}\phi}{\prod_{\beta\neq\alpha}(y_{\beta}-y_{\alpha})^{c_{\beta}}}.\] \label{DIgenpt}
\end{prop}

\begin{proof}
As in the proof of Theorem \ref{expordp}, we can evaluate our $D_{I}\phi$ by considering the ordinary $D_{I}\phi$ from Lemma \ref{DIexI} at points that are close to the desired one and in which all the variables take different values, and then take the limit at our point in a convenient way. For this we decompose $D_{I}\phi$ from Lemma \ref{DIexI} as the sum over $1\leq\alpha \leq s$ of $\sum_{i \in J_{\alpha}}\partial_{i}\phi\big/\prod_{i \neq j \in I}(x_{j}-x_{i})$, and in the part associated with $\alpha$ we first take the limit where all the variables $x_{i}$ with $i \in J_{\alpha}$ tend to $y_{\alpha}$, and then consider the limit in the other variables. Theorem \ref{expordp} shows that the $\alpha$th part becomes $\sum_{(a,\vec{b}) \in A_{I \setminus J_{\alpha},|J_{\alpha}|-1}}(-1)^{|J_{\alpha}|-a}\partial_{J}^{a}\phi\big/\prod_{k \in I \setminus J_{\alpha}}(x_{k}-y_{\alpha})^{b_{k}}$ at the first limit, and at the second limit, each denominator $\prod_{k \in I \setminus J_{\alpha}}(x_{k}-y_{\alpha})^{b_{k}}$ becomes $\prod_{\beta\neq\alpha}(y_{\beta}-y_{\alpha})^{c_{\beta}}$, with $c_{\beta}:=\sum_{k \in J_{\alpha}}b_{k}$.

Now, from the conditions on elements of $A_{I \setminus J_{\alpha},|J_{\alpha}|-1}$ we clearly have the equality $a+\sum_{\beta\neq\alpha}c_{\beta}=a+\sum_{k \in I \setminus J}=d=|I|$, and as $b_{k}\geq1$ for each $k \in I \setminus J_{\alpha}$ we get $c_{\beta}\geq|J_{\beta}|$ for every $\beta\neq\alpha$. The limit of the $\alpha$th part of $D_{I}\phi$ is thus indeed the sum over $(a,\vec{c}) \in B_{\vec{J},\alpha}$ of the corresponding function (with the asserted sign), times a coefficient counting how many elements $(a,\vec{b}) \in A_{I \setminus J_{\alpha},|J_{\alpha}|-1}$ that become our element $(a,\vec{c}) \in B_{\vec{J},\alpha}$ under the map taking such an element of $A_{I \setminus J_{\alpha},|J_{\alpha}|-1}$ to that of $B_{\vec{J},\alpha}$ having the same $a$ and for which $c_{\beta}$ is $\sum_{k \in J_{\alpha}}b_{k}$. But as the $b_{k}$'s are positive and free up to the sum conditions, we can consider the question for each $\beta$ separately, and by subtracting 1 from each we need the number of options to add $|J_{\beta}|$ ordered non-vanishing integers and get $c_{\beta}-|J_{\beta}|$, which is $\binom{c_{\beta}-1}{|J_{\beta}|-1}$ by a standard exercise in Discrete Mathematics (put $|J_{\beta}|-1$ sticks between $c_{\beta}-|J_{\beta}|$ balls to represent a solution, and count the options to do this). This proves the proposition.
\end{proof}

\smallskip

Using Proposition \ref{DIgenpt}, we can obtain the form of the action of our operators $D_{d}$ from Proposition \ref{Ddoneh} and Theorem \ref{coorddef}, and with them of the derivatives from Corollary \ref{difcoord}, at any point in $\mathbb{R}^{N}$. Take an index $1\leq\alpha \leq M$, and assume that we are given integers $c_{\beta}\geq0$ for all $\beta\neq\alpha$. We then denote by $\Xi_{\alpha,\vec{c}}$ the set of all integers $\{\kappa_{\beta}\}_{\beta\neq\alpha}$ such that $\kappa_{\beta}=0$ when $c_{\beta}=0$ and $1\leq\kappa_{\beta} \leq c_{\beta}$ if $c_{\beta}>0$.
\begin{prop}
Let now the full set $\mathbb{N}_{N}$ of indices be the disjoint union $\bigcup_{\alpha=1}^{M}H_{\alpha}$ of non-empty sets, and take a point where $x_{i}=y_{\alpha}$ for any variable $x_{i}$ associated with $i \in H_{\alpha}$, and that $y_{\alpha} \neq y_{\beta}$ when $\alpha\neq\beta$. Take some $1 \leq d \leq N$, and then we have the equality $D_{d}\phi$ is $d!$ times the sum over $1\leq\alpha \leq M$ of \[D_{d}\phi=d!\sum_{\alpha=1}^{M}\sum_{\sum_{\beta\neq\alpha}c_{\beta}<d}C_{\alpha,\vec{c}}\frac{\partial_{H_{\alpha}}^{d-\sum_{\beta\neq\alpha}c_{\beta}}\phi}{\prod_{\beta\neq\alpha}(y_{\beta}-y_{\alpha})^{c_{\beta}}},\] where the combinatorial multipliers are given by \[C_{\alpha,\vec{c}}:=\sum_{\vec{\kappa}\in\Xi_{\alpha,\vec{c}}}(-1)^{\sum_{\beta\neq\alpha}(c_{\beta}-\kappa_{\beta})}\binom{|H_{\alpha}|}{d-\sum_{\beta\neq\alpha}\kappa_{\beta}}\prod_{\substack{\beta\neq\alpha \\ c_{\beta}>0}}\binom{c_{\beta}-1}{\kappa_{\beta}-1}\binom{|H_{\beta}|}{\kappa_{\beta}}.\] \label{Ddallpt}
\end{prop}

\begin{proof}
Recall that $D_{d}\phi$ is $d!$ times the sum of $D_{I}\phi$ for subsets $I\subseteq\mathbb{N}_{N}$ of size $d$. Take such a set $I$, set $J_{\alpha}:=I \cap H_{\alpha}$ for every $\alpha$ (some of which might be empty), and Proposition \ref{DIgenpt} describes $D_{I}\phi$ in terms of the sets $\{J_{\alpha}\}_{\alpha=1}^{M}$ (ignoring the empty ones). Now, the values $y_{\alpha}$ and $y_{\beta}$ are the same from those coming from $H_{\alpha}$ and $H_{\beta}$, and the derivative $\partial_{J_{\alpha}}^{a}\phi$ equals $\partial_{H_{\alpha}}^{a}\phi$, so that the only remaining dependence on the sets $\{J_{\alpha}\}_{\alpha=1}^{M}$ are through its cardinalities. Therefore $D_{d}\phi$ is indeed a combination of the expressions $\partial_{H_{\alpha}}^{a}\phi\big/\prod_{\beta\neq\alpha}(y_{\beta}-y_{\alpha})^{c_{\beta}}$, where the equality $a+\sum_{\beta\neq\alpha}c_{\beta}=d$ holds, and where $\partial_{H_{\alpha}}^{a}\phi$ shows up in a term from $D_{I}\phi$ if and only if $J_{\alpha}\neq\emptyset$, while in this case $y_{\beta}-y_{\alpha}$ appears with a non-trivial power if and only if $J_{\beta}\neq\emptyset$ as well.

This means that for a given term, the sets $I$ for which $D_{I}\phi$ contributes to the coefficient of said term are those in which the sets $J_{\beta}$ that are non-empty are precisely those with $\beta=\alpha$ or with $y_{\beta}-y_{\alpha}$ showing up in the denominator, and by the definition of $B_{\vec{J},\alpha}$ in that proposition, we must have $|J_{\beta}| \leq c_{\beta}$ for every $\beta\neq\alpha$ as well as $a+\sum_{\beta\neq\alpha}=d=|I|=\sum_{\beta=1}^{M}|J_{\beta}|$. The inequality $|J_{\alpha}| \geq a$, required in Proposition \ref{DIgenpt}, indeed follows, and the exponent $|J_{\alpha}|-a$ in the sign from that proposition equals $\sum_{\beta\neq\alpha}(c_{\beta}-|J_{\beta}|)$. Therefore the cardinalities $\kappa_{\beta}$ of $J_{\beta}$ for all $\beta\neq\alpha$ (including the indices of empty sets) lie in the corresponding set $\Xi_{\alpha,\vec{c}}$, and for an element of this set the total contribution to the coefficient $C_{\alpha,\vec{c}}$ is the asserted sign times the product of the binomial coefficients $\binom{c_{\beta}-1}{\kappa_{\beta}-1}$ from Proposition \ref{DIgenpt}, times the number of options to choose $J_{\beta}$ from $H_{\beta}$ for every $\beta$ (including $\alpha$) in order to form $I$. As the choices for different $\beta$'s are independent, and there are $\binom{|H_{\beta}|}{\kappa_{\beta}}$ such options for every $\beta\neq\alpha$ (which means one option when $\kappa_{\beta}=0$, and we can discard those in the product), and the remaining binomial coefficient represents the number of choices of $J_{\alpha}$, of size $d-\sum_{\beta\neq\alpha}\kappa_{\beta}$, from $H_{\alpha}$, this is indeed the required contribution. Thus $C_{\alpha,\vec{c}}$ is the sum of those expressions over elements of $\Xi_{\alpha,\vec{c}}$, and we have the external multiplier $d!$ in the first equation. This proves the proposition.
\end{proof}
As for the combinatorial coefficients from Proposition \ref{Ddallpt}, note that the set $\Xi_{a,\vec{c}}^{\alpha}$ is a singleton if $c_{\beta}\leq1$ for every $\beta$, but not otherwise. In the former case the unique element is with $\kappa_{\beta}=c_{\beta}$ for all $\beta$ and hence $d-\sum_{\beta\neq\alpha}\kappa_{\beta}=a$, so that the coefficient $C_{a,\vec{c}}$ equals just $\binom{|H_{\alpha}|}{a}\prod_{\{\beta|c_{\beta}=1\}}|H_{\beta}|$ (this is, in particular, the case when $a=d$ and all the $c_{\beta}$'s vanish, where $\partial_{H_{\alpha}}^{a}\phi$ simply comes multiplied by $\binom{|H_{\alpha}|}{d}$). However, even in the next simplest case, where $c_{\gamma}=2$ and the other $c_{\beta}$'s are 0 and 1, the set $\Xi_{a,\vec{c}}^{\alpha}$ contains two elements, and $C_{a,\vec{c}}$ takes the value $\big[\binom{|H_{\alpha}|}{a}\binom{|H_{\gamma}|}{2}-\binom{|H_{\alpha}|}{a+1}|H_{\gamma}|\big]\prod_{\{\beta|c_{\beta}=1\}}|H_{\beta}|$, and the expression in parentheses expands as $\binom{|H_{\alpha}|}{a}|H_{\gamma}|\cdot\frac{(a+1)|H_{\gamma}|-2|H_{\alpha}|+(a-1)}{2a+2}$, which does not seem to simplify in any reasonable manner. When the $c_{\beta}$'s have larger values these coefficients are, of course, more complicated. Therefore in the generality that Proposition \ref{Ddallpt} is given we can write $D_{1}\phi$ as $\sum_{\alpha=1}^{M}|H_{\alpha}|\partial_{H_{\alpha}}^{1}\phi$, and we have \[D_{2}\phi=2\sum_{\alpha=1}^{M}\sum_{\beta\neq\alpha}|H_{\alpha}|\cdot|H_{\beta}|\cdot\frac{\partial_{H_{\alpha}}^{1}\phi}{y_{\beta}-y_{\alpha}}+\sum_{\alpha=1}^{M}|H_{\alpha}|(|H_{\alpha}|-1)\partial_{H_{\alpha}}^{2}\phi,\] but $D_{3}$ already contains coefficients like $3|H_{\alpha}|\cdot|H_{\beta}|(|H_{\alpha}|-|H_{\beta}|)$ in front of $\partial_{H_{\alpha}}^{1}\phi/(y_{\beta}-y_{\alpha})^{2}$. However, since the derivatives are local, at a given point we can find coordinates that respect only the partial symmetry existing around this point, which will produce a very simple form of all the derivatives of the function $\phi$ near that point---see Theorem \ref{coorgenpt} below.

For finding these coordinates, we first consider the special case of interest in Proposition \ref{Ddallpt} where the point lies along the total diagonal. Then we get the following consequence, which we phrase in terms of the normalized differential operators $\hat{D}_{r}$ from Remark \ref{Ntoinfty}.
\begin{cor}
At the point in the total diagonal of $\mathbb{R}^{N}$ where all the variables $x_{i}$, $i\in\mathbb{N}_{N}$ attain the same value $a$, the operator $\hat{D}_{d}$ from Remark \ref{Ntoinfty} takes $\phi$ to $\partial_{\mathbb{N}_{N}}^{d}\phi$. When $\phi$ is the trace function taking $(x_{1},\ldots,x_{N})$ to $\sum_{i\in\mathbb{N}_{N}}f(x_{i})$ for some function $f$, the value of $\hat{D}_{d}\phi$ is $\frac{(-1)^{d-1}}{(d-1)!}f^{(d)}(a)$. \label{totdiag}
\end{cor}

\begin{proof}
We apply Proposition \ref{Ddallpt} for the partition of $\mathbb{N}_{N}$ into one single set $H_{\alpha}=\mathbb{N}_{N}$. Then there are no $\beta\neq\alpha$, meaning that for the unique value of $\alpha$, we only get a multiple of $\partial_{\mathbb{N}_{N}}^{d}\phi$ in the formula for $D_{d}\phi$. Since the single set $H_{\alpha}$ is of size $N$, the multiplying coefficient $C_{d,\emptyset}$ reduces to $\binom{N}{d}$, and we have the multiplier $d!$. As the normalization from Remark \ref{Ntoinfty} gives $\hat{D}_{d}=\frac{(N-d)!D_{d}}{N!}$, the two multipliers cancel and the first assertion follows. When $\phi$ is a trace function, all the derivatives $\partial_{\mathbb{N}_{N}}^{\sigma}\phi$ for $\sigma \vdash d$ with $l(\sigma)\geq2$ vanish, so that the formula for $\partial_{\mathbb{N}_{N}}^{d}\phi$ in the definition in Theorem \ref{expordp} reduces to a single term, with the coefficient $(-1)^{d-1}/(d-1)!$. As the derivative itself is the pure one of order $d$ in some variable $x_{i}$, its value at our total diagonal point is $f^{(d)}(a)$. This proves the corollary.
\end{proof}
The simple formulae from Corollary \ref{totdiag} do not extend to compositions---see Remark \ref{nocomp} below.

\begin{rmk}
There might be different normalizations for the $u_{r}$'s and the $D_{d}$'s that might produce neater results in some situations. For example, if we replace $\hat{u}_{r}$ by $(-1)^{r-1}r\hat{u}_{r}$, then the limit from Proposition \ref{Ntoinfty} becomes simply $p_{r}$, and the coefficients in Theorem \ref{coorddef} will be integers, and co-prime ones (this is related to the Newton identities mentioned in the proof of Proposition \ref{Ntoinfty}, via Remark \ref{Bell}). Then the $g$th dual operator which will now be $\frac{(-1)^{g-1}}{g}\hat{D}_{g}$, will be defined along the total diagonal as in Theorem \ref{totdiag} by the formula from Theorem \ref{expordp} but without the $g$ in the numerator and with the sign $(-1)^{l(\sigma)-1}$, and the image of as trace function arising from $f$ under the $d$th operator will simply be the Taylor coefficient $f^{(d)}(a)/d!$. Alternatively, one could have taken $\frac{(-1)^{r-1}}{(r-1)!}\hat{u}_{r}$, so that the operator $(-1)^{d-1}(d-1)!\hat{D}_{d}$ will take a trace function along the total diagonal just to the $d$th derivative of the basic univariate function. However, the normalization that we initially chose for $D_{d}$ is the one with which the basic calculation in Corollary \ref{actnoreh} involves no coefficients, and $u_{r}$ are the dual coordinates in this normalization. \label{norm}
\end{rmk}

\smallskip

At a generic point in $\mathbb{R}^{N}$ the variables take different values. Therefore at a small neighborhood of that point one can distinguish the variables according to their values, and at such a point a much simpler way to obtaining a full set of $N$ derivatives at the point would simply be the ordinary partial ones with respect to the initial variables (i.e., the derivative obtained by fixing all the variables but one), regardless of the symmetry property. This was seen (e.g., in Lemma \ref{dersym}) not to give enough derivatives when some values coincide, but merging the results that we got with the latter simple observation gives the following, simpler choice of coordinates around any given point in $\mathbb{R}^{N}$.
\begin{thm}
Take some point in $\mathbb{R}^{N}$, and define the partition of $\mathbb{N}_{N}$ into sets $\{H_{\alpha}\}_{\alpha=1}^{M}$ such that $x_{i}=y_{\alpha}$ wherever $i \in H_{\alpha}$ and $y_{\alpha} \neq y_{\beta}$ when $\alpha\neq\beta$ as in Proposition \ref{Ddallpt}. For each such $\alpha$ we define $|H_{\alpha}|$ coordinates $\hat{u}_{r}^{|H_{\alpha}|}(x_{H_{\alpha}})$, with the variables being taken only from $H_{\alpha}$, and $\hat{u}_{r}$ being the functions from Proposition \ref{Ntoinfty}, with the normalizing coefficient there, as well as those in the expressions $\tilde{e}_{\lambda}$ showing up in Theorem \ref{difcoord} for these variables, naturally being based on $|H_{\alpha}|$ rather than the full number $N$ of variables. Then, expressing any function $\phi$ which is differentiable enough times around our point using the set of coordinates $\big\{\hat{u}_{r}^{|H_{\alpha}|}(x_{H_{\alpha}})\big\}_{1\leq\alpha \leq M,\ 1 \leq r\leq|H_{\alpha}|}$, the respective derivatives of our representing function of these coordinates are $\{\partial^{d}_{H_{\alpha}}\phi\}_{1\leq\alpha \leq M,\ 1 \leq d\leq|H_{\alpha}|}$. If $\phi$ is the trace function arising from the univariate function $f$ then for each $\alpha$ these derivatives are $\{(-1)^{d-1}f^{(d)}(y_{\alpha})/(d-1)!\}_{d=1}^{|H_{\alpha}|}$. \label{coorgenpt}
\end{thm}
Of course, replacing $\hat{u}_{r}^{|H_{\alpha}|}(x_{H_{\alpha}})$ by some other normalization, like the ones from Remark \ref{norm}, one can obtain, for a trace function, the Taylor coefficients $f^{d}(y_{\alpha})/d!$, or just the derivatives $f^{(d)}(y_{\alpha})$ themselves, for every $1 \leq d\leq|H_{\alpha}|$ (and every $1\leq\alpha \leq M$). The coordinates used in Lemma 2.4 of \cite{[EKZ]} are, in some other scalar normalization, the ones arising from Theorem \ref{coorgenpt} when the sizes from that theorem being either 1 or 2.

\begin{proof}
At any point in a small enough neighborhood of our point, we can distinguish the variables associated with $i \in H_{\alpha}$ for any $1\leq\alpha \leq M$ as those whose values are close enough to $y_{\alpha}$. We thus take such $\alpha$, fix all the variables $x_{i}$ with $i \not\in H_{\alpha}$ as the corresponding value $y_{\beta}$ (with $\beta\neq\alpha$), and consider the resulting function of $\{x_{i}\}_{i \in H_{\alpha}}$. It is a symmetric function of these variables, considered around the total diagonal point where they all equal $y_{\alpha}$. We therefore express this function using the coordinates $\big\{\hat{u}_{r}^{|H_{\alpha}|}(x_{H_{\alpha}})\big\}_{r=1}^{|H_{\alpha}|}$, and the respective derivatives are $\{\partial^{d}_{H_{\alpha}}\phi\}_{d=1}^{|H_{\alpha}|}$ (or the asserted derivatives of $f$ at $y_{\alpha}$ alone in case $\phi$ is the trace function associated with $f$) by Theorem \ref{totdiag}.

Now, this description of the derivatives of our function of $|H_{\alpha}|$ variables give derivatives of our function representing $\phi$ as long as the other coordinates that we take for $\phi$ do not involve the variables $\{x_{i}\}_{i \in H_{\alpha}}$. As our set of coordinates has the property that a coordinate with index $\alpha$ is only based on the variables $\{x_{i}\}_{i \in H_{\alpha}}$, gathering these coordinates together indeed gives a set of coordinates for $\phi$ in which the total set of derivatives is the union of the ones arising from each $\alpha$ separately. As this yields the asserted formulae (also in the trace function case), this proves the theorem.
\end{proof}

\smallskip

We now complete a missing proof from above.
\begin{proof}[Proof of Theorem \ref{uniqders}]
We already saw that $\{u_{r}\}_{r=1}^{N}$ generate, over $\mathbb{Q}$, all the symmetric functions in $\{x_{i}\}_{i\in\mathbb{N}_{N}}$. It follows, by homogeneity, that we can write every $v_{r}$ as a polynomial $p_{r}$ in $\{u_{d}\}_{d=1}^{r}$. Moreover, $p_{r}$ is of the form $c_{r}u_{r}$ plus a polynomial in $\{u_{d}\}_{d=1}^{r-1}$ (homogeneity again), and for the $v_{r}$'s to be coordinates, i.e., algebraically independent, we must have $c_{r}\neq0$ for every $r$.

We now write $D_{d}\phi$ as $\partial_{u_{d}}\varphi(u_{1},\ldots,u_{N})=\partial_{u_{d}}\eta(v_{1},\ldots,v_{N})$ via Corollary \ref{difcoord} and the fact that $\varphi(u_{1},\ldots,u_{N})$ and $\eta(v_{1},\ldots,v_{N})$ both represent the same expression $\phi(x_{1},\ldots,x_{N})$, and then the chain rule expresses the latter derivative as $\sum_{r=d}^{N}\frac{\partial p_{r}}{u_{d}}\cdot\eta_{v_{r}}$, with the summation starting from $r=d$ since $v_{r}=p_{r}(u_{1},...,u_{r})$ does not depend on $u_{d}$ with $d>r$. Moreover, $\frac{\partial p_{d}}{u_{d}}$ is the scalar $c_{d}\neq0$, meaning that the column vector with entries $\{D_{d}\phi\}_{d=1}^{N}$ is obtained from $\{\eta_{v_{r}}\}_{r=1}^{N}$ via multiplication by an invertible upper-triangular matrix $M$. Moreover, the entry $M_{dr}$ of $M$, with $r \geq d$, is the derivative $\frac{\partial p_{r}}{u_{d}}$, which is homogenous of degree $r-d$ in $\{x_{i}\}_{i\in\mathbb{N}_{N}}$.

Thus, to obtain the derivatives $\eta_{v_{r}}$, $1\ \leq r \leq N$ themselves, we multiply the vector of the $D_{d}\phi$'s by the inverse matrix $M^{-1}$. The degree $r-d$ homogeneity of $M_{dr}$ implies that the entry $M^{-1}_{rd}$ of $M^{-1}$, with $d \geq r$, is homogenous of degree $d-r$ (as so are all the combinations of products of entries of $M$ that go into this inverse entry). Writing the symmetric polynomial appearing in $M^{-1}_{rd}$ as $q_{d,r}(x_{1},\ldots,x_{n})$, we deduce that $\eta_{v_{r}}=\sum_{d=r}^{N}q_{d,r} \cdot D_{d}\phi$, where $q_{d,d}$ is the constant $\frac{1}{c_{d}}$.

Now, the function $q_{d,r}$ is symmetric of degree $d-r$, and when it multiplies $D_{d}$, the statement that $\eta_{v_{r}}$ has no non-constant numerator means that at every point in $\mathbb{R}^{N}$, and in every term in the presentation of $D_{d}$ at that point, the denominator from that presentation of $D_{d}$ has to cancel with the form that $q_{d,r}$ takes at that point. For every $0 \leq s \leq N-1$ consider the set $Y_{s}$ where all the variables $x_{i}$ with $i>s$ have to take the same value, and the variables $x_{i}$ with $1 \leq i \leq s$ are free. It is clear that $Y_{s} \subseteq Y_{s+1}$, that $Y_{0}$ is the total diagonal, and $Y_{N-1}$ is the full space. Assume that $q_{d,r}\neq0$ for some $r$ and $d>r$, take $s$ to be minimal such that some such $q_{d,r}$ does not vanish identically on $Y_{p}$ (the existence of such a $s$ is clear), and let $d$ be maximal such that $q_{d,r}$ does not vanish identically on this $Y_{s}$ for some $r<d$. For such an $r$ (yielding $s$ and $d$) we shall consider the expression for $\eta_{v_{r}}$ at a point in $Y_{s}$ such that all the variables $x_{i}$ with $1 \leq i \leq s+1$ take distinct values, meaning that in the partition from Proposition \ref{Ddallpt} we can take $M=s+1$, with $H_{i}$ being the singleton $\{i\}$ for $1 \leq i \leq s$, and $H_{s+1}$ containing all the other indices from $\mathbb{N}_{N}$, and is thus of size $N-s$. We also have $y_{i}=x_{i}$ for all $1 \leq i \leq s$, and we denote by $y$ the common value $y_{s+1}$ of all $x_{i}$ with $i \in H_{s}$, i.e., $i>s$.

We will be interested in the summands associated with $\alpha=s+1$ in Proposition \ref{Ddallpt}. Since $|H_{\beta}|=1$ for every $\beta\neq\alpha$, we only obtain contributions to $C_{\alpha,\vec{c}}$ from the element of $\Xi_{\alpha,\vec{c}}$ in which $\kappa_{\beta}=1$ wherever $c_{\beta}>0$ (and $\kappa_{\beta}=0$ when $c_{\beta}=0$ by definition). Assuming that $|\{1\leq\beta \leq s|c_{\beta}>0\}|$ is some number $1 \leq t \leq s$, this coefficient is a sign times $\binom{N-s}{d-t}$. For maximizing the order of the differential operator, we consider the summands in which every positive $c_{\beta}$ equals 1 (yielding a positive $C_{\alpha,\vec{c}}$ multiplying $\partial_{H_{s}}^{d-t}\phi$ divided by a product of distinct linear terms), and we also concentrate on the summands in which $t$ is minimal (but $C_{\alpha,\vec{c}}\neq0$).

Now, if $d \leq N-s$ then we can take $t=0$, and obtain a non-trivial numerator times a multiple of $\partial_{H_{s}}^{d}\phi$, contradicting our assumption on such numerators at every point. This covers the case where $s=0$. Otherwise we have $d=N-s+t$ for some $1 \leq t \leq s$, we take this value of $t$ (with $d-t=N-s$), and we find that the numerator $q_{d,r}$ has to cancel to a constant when divided by any product of $t$ of the expressions $x_{i}-y$ for $1 \leq i \leq s$. But in the polynomial ring in $x_{i}$, $1 \leq i \leq s$ (for fixed $y$) this means that $q_{d,r}$ is divisible by all this linear expressions, hence by their product, and thus if $t<s$ then yet again we obtain $\partial_{H_{s}}^{d-t}\phi=\partial_{H_{s}}^{N-s}\phi$ times a non-trivial numerator. It remains to consider the case where $t=s$, i.e., $d=N$, and then when we divide $q_{d,r}$ by the product $\prod_{i=1}^{s}(x_{i}-y)$ then we must get a constant. This determines the degree $d-r$ of $q_{d,r}$ to be $s$, and as $d=N$ this can only happen for $t=N-s$.

This function $q_{N,N-s}$ is symmetric in $N$ variables and homogenous of degree $s$, but reduces to $\prod_{i=1}^{s}(x_{i}-y)$ when $x_{i}=y$ for all $i>s$. But this function contains the monomial $\prod_{i=1}^{s}x_{i}$ with the coefficient 1, and the monomials $\prod_{i=1}^{s-1}x_{i} \cdot x_{j}$ with $j>s$ with coefficients that sum, over $j>s$, to $-1$. As this function cannot be symmetric in all $N$ variables, we reach a contradiction, meaning that $q_{d,r}$ must vanish for every $d>r$ under the hypothesis of no non-constant numerators. This means that our matrix $M^{-1}$, and with it $M$, are diagonal, and therefore the polynomial $p_{r}$ does not depend on $u_{d}$ for any $d<r$. But this implies that $v_{r}=c_{r}u_{r}$, and as the reciprocal $\frac{1}{c_{r}}$ on the diagonal is the coefficient with which $D_{r}\phi$ shows up in $\eta_{v_{r}}$, our assumption on that coefficient yields $c_{r}=1$ for every $1 \leq r \leq N$. This implies that $v_{r}=u_{r}$ for every such $r$ as desired. This completes the proof of the theorem.
\end{proof}

\begin{rmk}
One may ask about composing the operators $D_{d}$ along the total diagonal, for obtaining, for example a shorter proof for Theorem \ref{annbyders}. Indeed, we can define compositions $D_{\lambda}=\prod_{h=1}^{d}D_{h}^{m_{h}}$ for some partition $\lambda \vdash d$ (with $m_{h}$ being the multiplicity of $h$ in $\lambda$ as always), and they are all linearly independent by Corollary \ref{Weyl}. Then applying such a $D_{\lambda}$ to $u_{r}$ when $d<r$ would annihilate it (just use some $D_{h}$ with $m_{h}>0$ and the fact that $D_{h}u_{r}=0$ since $h<r$), and in case one could replace, along the total diagonal, every $D_{h}$ by the formula from Theorem \ref{totdiag}, then the $D_{\lambda}$'s for $\lambda \vdash d$ would have given linearly independent operators of order $d$ there, using which the proof of Theorem \ref{annbyders} would become much easier. To see that this argument is wrong, we consider the case $d=r$, and Proposition \ref{Ntoinfty}. Such an argument would have implied that the only such derivative $D_{\lambda}$ that does not annihilate $u_{r}$ would have been the one with $l(\lambda)=1$, and thus the derivatives from that proposition, which are constants and therefore can be evaluated, in particular, along the total diagonal, would be based on the coefficients in the presentation of these operators in terms of the $D_{\lambda}$'s, which do not depend on $N$. The much finer dependence on $N$ of the values from the proof of that proposition exemplifies that this is not the case, and an argument using the compositions of the derivatives from Theorem \ref{totdiag} for compositions of the $D_{d}$'s is not valid. The reason is that for compositions of differential operators, we need their formulae on full neighborhoods, and the formulae from Theorem \ref{totdiag} are valid only at total diagonal points (indeed, comparing the values of the higher-order derivatives from that theorem with those of $D_{d}$ on general symmetric functions produces different values, due to the latter operating via Leibniz' Rule on products, which the former involve more expressions). \label{nocomp}
\end{rmk}

\smallskip

We conclude with a question about more general settings. Our coordinates are those for symmetric powers of the affine line, which are the invariant spaces of the action of $S^{N}$ on the usual $N$th power of the affine line. One may ask for other group actions on affine spaces, and good coordinates for differentiation on their spaces of invariants. Note that in many natural examples, like the symmetric powers of an affine space of dimension $M>1$ (where $S_{N}$ acts on a space of dimension $MN$), or the alternating group $A_{N}$ acting on such a space (now with $M\geq1$), the quotient spaces tend to be singular, including the GIT quotients (since the algebra of invariants is no longer a polynomial ring, like Theorem \ref{egen} establishes for the case considered in this paper). We leave such questions for further research.

\noindent\textsc{Einstein Institute of Mathematics, the Hebrew University of Jerusalem, Edmund Safra Campus, Jerusalem 91904, Israel}

\noindent E-mail address: zemels@math.huji.ac.il

\end{document}